\documentclass[a4paper,11pt]{amsart}

\title{Automorphism groups of randomized structures}
\author{Tom\'as Ibarluc\'ia}
\address{Universit\'e de Lyon \\Institut Camille Jordan \\43 blvd. du 11 novembre 1918\\69622 Villeurbanne Cedex \\France \\ \href{mailto:ibarlucia@math.univ-lyon1.fr}{\tt ibarlucia@math.univ-lyon1.fr}}

\usepackage{amsmath, amssymb, amsthm, amscd, dsfont}

\usepackage[hidelinks]{hyperref}
\makeatletter
\g@addto@macro\bfseries{\boldmath}
\makeatother

\usepackage[T1]{fontenc}
\usepackage{kpfonts}

\usepackage{stmaryrd}

\usepackage{geometry}
\def\mbA{\mathbb{A}}

\def\mbE{\mathbb{E}}
\def\mbP{\mathbb{P}}
\def\mbN{\mathbb{N}}
\def\Reg{\mathfrak{R}}
\def\mcM{\mathcal{M}}
\def\mcS{\mathcal{S}}
\def\mcK{\mathcal{K}}
\def\mcH{\mathcal{H}}
\DeclareMathOperator{\Aut}{Aut}
\DeclareMathOperator{\End}{End}
\DeclareMathOperator{\Iso}{Iso}

\newcommand{\actson}{\curvearrowright}

\def\acl{{\rm\text{acl}}}
\def\tp{{\rm\text{tp}}}

\def\Ind#1#2{#1\setbox0=\hbox{$#1x$}\kern\wd0\hbox to 0pt{\hss$#1\mid$\hss}
\lower.9\ht0\hbox to 0pt{\hss$#1\smile$\hss}\kern\wd0}

\def\Notind#1#2{#1\setbox0=\hbox{$#1x$}\kern\wd0\hbox to 0pt{\mathchardef
\nn="3236\hss$#1\nn$\kern1.4\wd0\hss}\hbox to 0pt{\hss$#1\mid$\hss}\lower.9\ht0
\hbox to 0pt{\hss$#1\smile$\hss}\kern\wd0}

\theoremstyle{plain}        \newtheorem{fact}{Fact}[section]
\theoremstyle{plain}        \newtheorem{theorem}[fact]{Theorem}
\theoremstyle{plain}        \newtheorem{lem}[fact]{Lemma}
\theoremstyle{plain}        \newtheorem{prop}[fact]{Proposition}
\theoremstyle{plain}        \newtheorem{cor}[fact]{Corollary}
\theoremstyle{definition}   \newtheorem{rem}[fact]{Remark} 
\theoremstyle{definition}   \newtheorem{defin}[fact]{Definition}
\theoremstyle{definition}   \newtheorem{notation}[fact]{Notation}
\theoremstyle{definition}   
\theoremstyle{definition}   \newtheorem{example}[fact]{Example}
\theoremstyle{definition}   \newtheorem{question}[fact]{Question}

\begin{document}
\begin{abstract} We study automorphism groups of randomizations of separable structures, with focus on the $\aleph_0$-categorical case. We give a description of the automorphism group of the Borel randomization in terms of the group of the original structure. In the $\aleph_0$-categorical context, this provides a new source of Roelcke precompact Polish groups, and we describe the associated Roelcke compactifications. This allows us also to recover and generalize preservation results of stable and NIP formulas previously established in the literature, via a Banach-theoretic translation. Finally, we study and classify the separable models of the theory of beautiful pairs of randomizations, showing in particular that this theory is never $\aleph_0$-categorical (except in basic cases).\end{abstract}
\maketitle
\tableofcontents

\section*{Introduction}

A randomization of a structure $M$ is a metric structure whose elements are random variables taking values in $M$, and whose predicates account for the expected values of the original predicates of $M$. The idea goes back to Keisler \cite{kei99}, where it was developed in a classical first-order framework. Later, in \cite{benkei}, Keisler and Ben Yaacov adapted the construction so that randomizations could be regarded as \emph{metric} structures (that is, in the sense of continuous first-order logic), and with this approach they proved several preservation results, supporting the claim that this is the correct frame to develop the idea.

The construction was further adapted by Ben Yaacov in \cite{benOntheories}, so that randomizations of metric structures could also be considered. In \cite{benVapChe}, another important and difficult preservation result was proved, concerning NIP formulas. Further model-theoretic analysis of randomized structures have been carried out in \cite{andgolkeiDef,andkeiSep,andgolkeiInd}.

In the present work, we approach the subject from the viewpoint of descriptive set theory. Our main motivation is the study of the symmetries of randomized structures. More precisely, we describe and study the automorphism group of the \emph{Borel randomization} of a separable metric structure. This is the most basic example of a randomization, yet encompassing most of the intuition of the subject. If $M$ is a separable structure with automorphism group~$G$, then the automorphism group of its Borel randomization $M^R$ is a \emph{measurable wreath product}, $$G\wr\Omega\coloneqq L^0(\Omega,G)\rtimes\Aut(\Omega),$$ where $\Omega$ denotes a standard probability space. For instance, the symmetry group of a randomized countable set, $\mbN^R$, is the semidirect product of the random permutations of $\mbN$ and the measure-preserving transformations of $\Omega$.

The group $G\wr\Omega$, induced by a given Polish group $G$, is an interesting object in itself. It had already been considered by Kechris in \cite[\textsection 19ff.]{kechrisGlobal}. Here, we investigate in detail the continuous actions of $G\wr\Omega$ induced by actions of~$G$. First, in Section~\ref{s:actions on metric}, we study the isometric actions of $G\wr\Omega$ of this kind. We show that every approximately oligomorphic faithful action of $G$ induces an approximately oligomorphic faithful action of $G\wr\Omega$. In particular, if $G$ is \emph{Roelcke precompact}, then so is $G\wr\Omega$. Afterwards, we draw a connection with the automorphism groups of randomized structures, as explained above.

Later, in Section~\ref{s:rando of ambits}, we investigate some compact $G\wr\Omega$-flows. This corresponds to the study of type spaces in randomized structures. When $G$ is Roelcke precompact, we give an explicit description of the Roelcke compactification of $G\wr\Omega$, denoted by $R(G\wr\Omega)$, in terms of $R(G)$. Furthermore, we show how some properties $R(G)$ might have (existence of a compatible semigroup law, representability by contractions on Hilbert spaces) pass on to $R(G\wr\Omega)$. We also prove a general preservation result concerning Banach representations of randomized type spaces.

If $M$ is the separable model of an $\aleph_0$-categorical theory $T$, then most of the model-theoretic information of $T$ is coded by dynamical properties of $G=\Aut(M)$, as per the recent works \cite{bentsa,iba14,bit15}. On the other hand, the randomized theory $T^R$ is also $\aleph_0$-categorical; hence, in this case, the Borel randomization encompasses all the model-theoretic information of $T^R$, and this can be recovered from the group $G\wr\Omega$. The preservation results of Section~\ref{s:rando of ambits} get then a precise model-theoretic meaning, and allow us to give new proofs (in the $\aleph_0$-categorical setting) of the theorems of preservation of stability and NIP from \cite{benkei,benOntheories} and \cite{benVapChe}.

Finally, in Section~\ref{s:beautiful}, we come back to more model-theoretic concerns, and we study the theory $(T^R)_P$ of beautiful pairs of models of a randomized theory $T^R$. This is motivated by the results of \cite{bbhSFB14} on the problem of generalizing the notion of one-basedness to the metric setting. When $T$ is $\aleph_0$-categorical, we classify the separable models of $(T^R)_P$, and show in particular that $(T^R)_P$ is never $\aleph_0$-categorical (except when $T$ is the theory of a compact structure). We also extend to the metric setting the result of preservation of $\aleph_0$-stability from \cite{benkei}. We end with a description of the automorphism groups of some canonical models of $(T^R)_P$.

\medskip \noindent \textbf{Acknowledgements.}
I am grateful to Ita\"i Ben Yaacov, who posed the question that was at the origin of this work. I also thank him and Julien Melleray for valuable discussions.

Part of the present work was realized during a research visit to the logic group of the Universidad de los Andes in 2014, supported by the ECOS Nord exchange programme. I would like to thank the Universidad de los Andes for this enriching stay, and Alexander Berenstein for very stimulating conversations.

This research was partially supported by the Agence Nationale de la Recherche project GruPoLoCo (ANR-11-JS01-0008).

\noindent\hrulefill

\section{Preliminaries}

\subsection{Notation}

Throughout the paper we fix an atomless Lebesgue space $\Omega$, say the unit interval $[0,1]$ with the Lebesgue measure $\mu$.

If $(X,d)$ is a Polish, bounded, metric space, then $$X^\Omega\coloneqq L^0(\Omega,X)$$ will denote the space of random variables $r\colon\Omega\to X$ (where $X$ carries the Borel $\sigma$-algebra), up to equality almost everywhere, endowed with the induced $L^1$-distance, $$d^\Omega(r,s)\coloneqq\int d(r(\omega),s(\omega))d\mu(\omega).$$ If $X$ is just a topological Polish space, then we may choose any compatible, bounded distance $d$ and define $X^\Omega$ to be the space $L^0(\Omega,X)$ with the topology induced by $d^\Omega$. This is independent of the choice of $d$: a sequence $r_n$ converges to $r$ in $X^\Omega$ if and only if every subsequence of $r_n$ has a further subsequence that converges almost surely to $r$. In both cases, as a metric or a topological space, $X^\Omega$ is Polish. See, for instance, \cite[Annexe~C]{lemaitreThese} or \cite[\textsection 19]{kechrisGlobal}.

An important particular case is the metric space $X=[0,1]$. Throughout the paper, we denote $$\mbA\coloneqq [0,1]^\Omega=L^0(\Omega,[0,1]).$$ Thus, the metric on $\mbA$ is given by $\mbE|R-S|$, where $\mbE\colon\mbA\to [0,1]$ is the expectation function $\mbE R=\int R(\omega)d\mu(\omega)$.

In the previous constructions, $\Omega$ might be replaced, later, by the unit square $\Omega^2$. We remark that $X^{\Omega_0\times\Omega_1}$ is naturally isomorphic to $(X^{\Omega_1})^{\Omega_0}$ via the natural map $r\mapsto\tilde{r}$, $\tilde{r}(\omega_0)(\omega_1)=r(\omega_0,\omega_1)$. For a proof, see \cite[418R]{fremVol4}.

Given a compact metrizable space $K$, we will denote by $\Reg(K)$ the compact space of Borel probability measures on $K$. The topology is the weak$^*$ topology as a subset of the dual space of continuous functions on $K$.

\subsection{A measurable wreath product}

We denote by $\Aut(\Omega)$ the group of invertible mea\-sure-preserving transformations of $\Omega$, up to equality almost everywhere. If $X$ is a Polish space, then $\Aut(\Omega)$ acts on $X^\Omega$ by the formula $$(tr)(\omega)=r(t^{-1}(\omega)),$$ where $t\in\Aut(\Omega)$ and $r\in X^\Omega$. If $(X,d)$ is metric, then this action is by isometries on $(X^\Omega,d^\Omega)$. In particular, $\Aut(\Omega)$ can be seen as a subgroup of the isometry group of $\mbA$, which is a Polish group under the topology of pointwise convergence. With the induced topology, $\Aut(\Omega)$ is a Polish group. In addition, the action $\Aut(\Omega)\actson X^\Omega$ is continuous.

Given a Polish group $G$, the space $G^\Omega$ is also a Polish group with the operation of pointwise multiplication. Hence we have an action of $\Aut(\Omega)$ on the group $G^\Omega$. In this case, given $g\in G^\Omega$ and $t\in\Aut(\Omega)$, we will denote the action of $t$ on $g$ by $t\cdot g$. (When $\Aut(\Omega)$ and $G^\Omega$ act simultaneously on a space $Z$, the term $tg$ will denote their product as homeomorphisms of $Z$.)

We introduce the following definition.

\begin{defin}
The \emph{measurable wreath product} of $G$ and $\Aut(\Omega)$ is the semidirect product $$G\wr\Omega\coloneqq G^\Omega\rtimes\Aut(\Omega),$$ which is a Polish group endowed with the product topology.
\end{defin}

If $G$ acts continuously on a Polish space $X$, then $G^\Omega$ acts continuously on $X^\Omega$, by the formula $$(gr)(\omega)=g(\omega)(r(\omega)),$$ where $g\in G^\Omega$ and $r\in X^\Omega$. Note that $gr$ is indeed a random variable: if $U\subset X$ is open and $\{U_i\}_{i\in\mbN}$ is a countable base for the topology of $X$, then $$(gr)^{-1}(U)=\bigcup_{i\in\mbN}\big(r^{-1}(U_i)\cap g^{-1}(\{h\in G:h(U_i)\subset U\})\big),$$ which is a measurable set since $\{h\in G:h(U_i)\subset U\}=\{h\in G:h^{-1}(U^c)\subset U_i^c\}$ is a closed subset of $G$. That the action $G^\Omega\actson X^\Omega$ is continuous can be deduced from the fact that $g_nr_n$ converges almost surely to $gr$ if $g_n$ and $r_n$ converge almost surely to $g$ and $r$. If moreover the action $G\actson (X,d)$ is by isometries, then so is $G^\Omega\actson (X^\Omega,d^\Omega)$.

Now, given a continuous action $G\actson X$, we have, simultaneously, continuous actions of $\Aut(\Omega)$ and $G^\Omega$ on $X^\Omega$. These induce a continuous action $G\wr\Omega\actson X^\Omega$. Indeed, inside the group of homeomorphisms of $X^\Omega$, the elements $t\in\Aut(\Omega)$, $g\in G^\Omega$ satisfy the relation $t\cdot g=tgt^{-1}$.

We say that an action by isometries $G\actson (X,d)$ is \emph{faithful} if the corresponding group homomorphism $G\to\Iso(X)$ is a topological embedding. Here, $\Iso(X)$ is the isometry group of $X$ with the topology of pointwise convergence.

\begin{lem}\label{faithful actions} Let $G$ be a Polish group acting continuously by isometries on a Polish metric space $(X,d)$ with at least two elements. If the action is faithful, then so is the induced action $G\wr\Omega\actson (X^\Omega,d^\Omega)$.\end{lem}
\begin{proof}
The action $\Aut(\Omega)\actson (X^\Omega,d^\Omega)$ is faithful since $X$ has at least two elements (we omit the details). Since the action $G\actson (X,d)$ is faithful, we see that a sequence in $G^\Omega$ converges to the identity if and only if every subsequence has a further subsequence $g_n$ such that, almost surely, $g_n(\omega)(x)$ converges to $x$ for every $x\in X$. Let $D\subset X$ be a countable dense subset and let $C\subset X^\Omega$ be the family of constant random variables taking a value from $D$. It follows that $g_n\to 1$ in $G^\Omega$ if and only if $d^\Omega(g_nc,c)\to 0$ for every $c\in C$. In particular, the action $G^\Omega\actson X^\Omega$ is faithful.

To see that the action $G\wr\Omega\actson X^\Omega$ is faithful we must check that, whenever $g_nt_n\to 1$ in $\Iso(X^\Omega)$ for $g_n\in G^\Omega$ and $t_n\in\Aut(\Omega)$, we have $g_n\to 1$ and $t_n\to 1$. Now, if $d^\Omega(g_nt_nr,r)\to 0$ for every random variable $r$, specializing on the constants $c\in C$ (which are fixed under $\Aut(\Omega)$) we see that $g_n\to 1$. Then also $t_n\to 1$.
\end{proof}

\subsection{Semigroup completions}

Let $G$ be a Polish group, and let $d_L$ be a left-invariant compatible metric on $G$, which exists by the Birkhoff--Kakutani theorem (see \cite[p.~28]{berberian}). The completion of $G$ with respect to $d_L$ will be denoted by $\widehat{G}_L$. Then, the group law on $G$ extends to a jointly continuous semigroup operation on $\widehat{G}_L$. As a topological semigroup, $\widehat{G}_L$ does not depend on the particular choice of $d_L$ (it is the completion of $G$ with respect to its \emph{left uniformity}). Similarly, the completion of $G$ with respect to a compatible right-invariant metric $d_R$ will be denoted by $\widehat{G}_R$, and is also a topological semigroup. The inverse operation on $G$ extends to a homeomorphic anti-isomorphism $^*\colon\widehat{G}_R\to\widehat{G}_L$.

The right completion of the group $\Aut(\Omega)$ is the semigroup $\End(\Omega)$ of measure-pre\-serv\-ing transformations of $\Omega$, up to equality almost everywhere (we will revisit this fact later). Then, given any Polish metric space $(X,d)$, the left completion $\End(\Omega)^*$ acts by isometries on $(X^\Omega,d^\Omega)$ by the formula $(s^*r)(\omega)=r(s(\omega))$, where $s\in\End(\Omega)$, $r\in X^\Omega$.

\begin{lem}\label{left completion of GwrOmega}
Let $G$ be a Polish group. Then, the left completion of $G\wr\Omega$ is the topological semigroup $(\widehat{G}_L)^\Omega\rtimes\End(\Omega)^*$.
\end{lem}
\begin{proof}
Fix a compatible left-invariant metric $d_L$ on the left completion $\widehat{G}_L$, and consider the induced metric $d_L^\Omega$ on $(\widehat{G}_L)^\Omega$. Since $(\widehat{G}_L,d_L)$ is complete, so is $((\widehat{G}_L)^\Omega,d_L^\Omega)$ (see \cite{kechrisGlobal}, Proposition~19.6). Let $d_\Omega$ be a compatible left-invariant metric on $\End(\Omega)^*$. Then, the metric $d=d_L^\Omega+d_\Omega$ is a complete, left-invariant metric on $(\widehat{G}_L)^\Omega\rtimes\End(\Omega)^*$, compatible with the topology of $G\wr\Omega$. Since $G\wr\Omega$ is moreover dense in $(\widehat{G}_L)^\Omega\rtimes\End(\Omega)^*$, this must be its left completion.
\end{proof}

\subsection{Borel randomizations}

Given a structure or a class of models of a certain first-order theory, there are several ways of producing randomizations from them. For the most part of this paper, however, we will only be interested in one particular construction, which can be considered the basic, canonical example of a randomization. In the $\aleph_0$-categorical setting, which is of particular interest to us, this is actually the one and only example one needs to consider, since randomizations preserve separable categoricity.

Unless otherwise stated, all structures and theories we consider are in the sense of first-order \emph{continuous} logic, as per \cite{benusv10,bbhu08}. (Traditional discrete structures and theories, which form a particular case, are refer to as \emph{classical}.) In particular, structures are complete metric spaces. Furthermore, we assume all our structures to be separable in a countable language.

Let $M$ be a metric structure with at least two elements, in a language $L$, which we shall assume to be one-sorted for simplicity. We introduce below the \emph{Borel randomization of $M$}, denoted in this paper by $M^R$, which is a structure in a two-sorted language $L^R$. We have borrowed the name from \cite{andkeiSep}, Definition~2.1, although, there, the term refers to the natural pre-structure whose completion gives $M^R$; also, they only define it for classical $M$.

The \emph{main sort} of $M^R$ is the metric space $(M^\Omega,d^\Omega)$, and the \emph{auxiliary sort} of $M^R$ is the space $\mbA$ with its natural metric. For each definable predicate $\varphi\colon M^n\to [0,1]$ there is a definable function $$\llbracket \varphi(x)\rrbracket\colon (M^\Omega)^n\to\mbA,$$ given by $$\llbracket\varphi(r)\rrbracket(\omega)=\varphi(r(\omega)),$$ where $r\in (M^\Omega)^n\simeq (M^n)^\Omega$. In addition, the auxiliary sort $\mbA$ is equipped with a predicate for the expectation function $$\mbE\colon\mbA\to [0,1],$$ and with definable functions for the basic arithmetic operations between random variables, which in fact permit to define all continuous pointwise-defined functions $\mbA^n\to\mbA$ (\cite{benOntheories}, Lemma~2.13). Thus, as a reduct, $\mbA$ is a model of the \emph{theory of $[0,1]$-valued random variables}, in the sense of \cite[\textsection 2]{benOntheories}.

For a syntactical presentation and an explicit description of the language $L^R$, see \cite[\textsection 3]{benOntheories} or, in the classical setting, \cite[\textsection 2]{benkei}. Let $T$ denote the first-order theory of $M$ in the language $L$. Then, the \emph{randomization theory} $T^R$ is the theory of $M^R$ in the language~$L^R$.

\begin{prop}
The theory $T^R$ has quantifier elimination. More precisely, if $x$ and $y$ are tuples from the main and the auxiliary sort, respectively, then the $T^R$-type of $xy$ is determined by the values $\mbE\tau(x,y)$ where $\tau(x,y)\in\mbA$ is a term on $xy$, that is, the result of applying any operations of the auxiliary sort to any random variables from $y$ and any random variables of the form $\llbracket\varphi(x)\rrbracket$ for an $L$-formula $\varphi$.
\end{prop}
\begin{proof}
See \cite[\textsection 3.5]{benOntheories}.
\end{proof}

One could present the Borel randomization $M^R$ as a structure in the sort $M^\Omega$ alone, by considering as predicates the functions $\mbE\llbracket\varphi(x)\rrbracket$. However, $\mbA$ would then be present as an imaginary sort (and a very important one, which justifies to make it a sort in its own right). Remark, in this respect, that the structure MALG$_\mu$, the measure algebra of~$\Omega$ (which can be thought of as $2^\Omega$), is bi-interpretable with~$\mbA$ (see \cite[\textsection 2]{benOntheories}).

\subsection{Bochner spaces}\label{ss:Bochner} Let $V$ be a Banach space, which we will assume to be separable. The \emph{Bochner space} $L^2(\Omega,V)$ is the space of measurable functions $f\colon\Omega\to V$, modulo equality almost everywhere, for which the norm $$\|f\|_2\coloneqq\big(\int\|f(\omega)\|^2_Vd\omega\big)^{1/2}$$ is finite. Equipped with this norm, $L^2(\Omega,V)$ is a Banach space.

If $G$ is a Polish group and $G\actson V$ is a continuous action by isometries, then the action $G^\Omega\actson V^\Omega$ (defined topologically, since the norm metric on $V$ is not bounded) restricts to an isometric action on $L^2(\Omega,V)$. Similarly for the action of $\Aut(\Omega)$. Thus, we obtain an isometric continuous action $G\wr\Omega\actson L^2(\Omega,V)$.

We recall a description of the dual space $L^2(\Omega,V)^*$. This can be identified with the space $L^2_{w^*}(\Omega,V^*)$ of weak$^*$-measurable functions $\psi\colon\Omega\to V^*$, modulo equality almost everywhere, for which there exists $h\in L^2(\Omega)$ with $\|\psi(\omega)\|_{V^*}\leq h(\omega)$ for almost every $\omega$. There is a natural linear map $L^2_{w^*}(\Omega,V^*)\to L^2(\Omega,V)^*$, defined implicitly by the relation $$\langle f,\psi\rangle=\int\langle f(\omega),\psi(\omega)\rangle d\omega$$ for $f\in L^2(\Omega,V)$ and $\psi\in L^2_{w^*}(\Omega,V^*)$. This is in fact a bijection; see \cite[\textsection 1.5]{cemmenBook}. We endow the space $L^2_{w^*}(\Omega,V^*)$ with the weak$^*$ topology as the dual of $L^2(\Omega,V)$.

Every isometric continuous action $G\actson V$ induces a dual action $G\actson V^*$, given by $\langle v,g\psi\rangle=\langle g^{-1}v,\psi\rangle$ for $v\in V$ and $\psi\in V^*$. This action is continuous for the weak$^*$ topology on $V^*$. In particular, we have a continuous action $G\wr\Omega\actson L^2_{w^*}(\Omega,V^*)$, which satisfies the relation $(gt\psi)(\omega)=g(\omega)(\psi(t^{-1}(\omega)))$ for $g\in G^\Omega$, $t\in\Aut(\Omega)$ and $\psi\in L^2_{w^*}(\Omega,V^*)$.

\noindent\hrulefill

\section{Actions on randomized metric spaces}\label{s:actions on metric}

\subsection{Quotients of isometric actions}

If $G\actson (X,d)$ is an action by isometries, we define the \emph{metric quotient} of $X$ by $G$ as the space of orbit closures $$X\sslash G\coloneqq\{\overline{Gx}:x\in X\}$$ endowed with the distance $d(\overline{Gx},\overline{Gy})=\inf_{g\in G}d(x,gy)$. If $(X,d)$ is a Polish metric space, then so is $(X\sslash G,d)$. The action $G\actson (X,d)$ is \emph{approximately oligomorphic} if the metric quotient $X^n\sslash G$ of the diagonal action $G\actson (X^n,d)$ is compact for all $n\in\mbN$. Equivalently, if the metric quotient $X^\mbN\sslash G$ is compact. Here, the metric $d$ on $X^n$ or $X^\mbN$ is any compatible distance for which the diagonal action of $G$ is isometric.

A Polish group $G$ is \emph{Roelcke precompact} if it admits a faithful approximately oligomorphic action on a Polish metric space $(X,d)$ (this is actually an equivalent property, we will recall the original definition in \textsection \ref{ss:Markov}). For example, the group $\Aut(\Omega)$ is Roelcke precompact and the action $\Aut(\Omega)\actson\mbA$ is approximately oligomorphic. See \cite{bentsa,benOntheories}.

Let $\pi\colon X\to Y$ be a surjective map between topological spaces. A \emph{Borel selector} for $\pi$ is a measurable map $\sigma\colon Y\to X$ such that $\pi\sigma$ is the identity of $Y$.

\begin{lem}\label{Borel selector}\ 
\begin{enumerate}
\item\label{B.s. metric quotient}
Let $G\actson (X,d)$ be an action by isometries on a Polish metric space $X$. Then, the quotient map $\pi\colon X\to X\sslash G$ admits a Borel selector.
\item\label{B.s. compact spaces} Let $\pi\colon K_0\to K_1$ be a continuous surjective map between compact metrizable spaces. Then $\pi$ admits a Borel selector.
\end{enumerate}
\end{lem}
\begin{proof}
(1). We adapt the proof of Theorem~12.16 from \cite{kechrisDST}. Let $F(X)$ be the \emph{Effros Borel space} of closed subsets of $X$ with the $\sigma$-algebra generated by the sets $[U]=\{F\in F(X):F\cap U\neq\emptyset\}$, where $U$ varies over the open subsets of $X$. As a set, the quotient $X\sslash G$ is contained in $F(X)$. Moreover, we have $[U]\cap (X\sslash G)=\pi(U)$, which is an open subset of $X\sslash G$ if $U$ is open in $X$. Indeed, let $u\in U$ and take $\epsilon>0$ such that $d(u,v)<\epsilon$ implies $v\in U$. Then, if $d(\pi(u),\pi(x))<\epsilon$, there is $g\in G$ such that $d(u,gx)<\epsilon$, whence $\pi(x)=\pi(gx)\in\pi(U)$. It follows that $A\cap (X\sslash G)$ is a Borel subset of $X\sslash G$ whenever $A$ is Borel in $F(X)$.

Now, by Theorem~12.13 in \cite{kechrisDST}, there is a measurable map $d\colon F(X)\to X$ such that $d(F)\in F$ for every non-empty closed set $F\subset X$. The restriction of $d$ to $X\sslash G$ is thus a Borel selector for $\pi$.

(2). As before, we know that there is a measurable map $d\colon F(K_0)\to K_0$ with $d(F)\in F$ for $F\neq\emptyset$, thus it suffices to show that the fiber map $\pi^{-1}\colon K_1\to F(K_0)$ is measurable. Then again, if $F\subset K_0$ is any subset, we have $(\pi^{-1})^{-1}([F])=\pi(F)$. Now, the $\sigma$-algebra of $F(K_0)$ is also generated by the sets $[F]$ where $F$ varies over the closed subsets of $K_0$. Since $K_0$ is compact and $\pi$ is continuous, $\pi(F)$ is closed if $F$ is closed, so the fiber map is measurable.
\end{proof}

Remark that if $G$ has a normal Polish subgroup $N$, then any isometric action $G\actson (X,d)$ induces an isometric action $G\actson (X\sslash N,d)$, by the formula $g\overline{Nx}=\overline{Ngx}$.

\begin{lem}\label{algebra of quotients} Let $G$ be a Polish group and let $G\actson (X,d)$ be an action by isometries on a Polish metric space. Suppose $G$ can be written as a product $G=NH$ for Polish subgroups $N,H<G$ with $N$ normal in $G$. Then we have the following natural isometric isomorphisms:
\begin{enumerate}
\item $X\sslash G\simeq (X\sslash N)\sslash H$.
\item $X^\Omega\sslash G^\Omega\simeq (X\sslash G)^\Omega$.
\item $X^\Omega\sslash (G\wr\Omega)\simeq (X\sslash G)^\Omega\sslash\Aut(\Omega)$.
\end{enumerate}
\end{lem}
\begin{proof}
(1).\ We verify that the map $\overline{Gx}\mapsto\overline{H\overline{Nx}}$ is isometric. Indeed, $\inf_{h\in H}d(\overline{Nx},h\overline{Ny})=\inf_{h\in H}d(\overline{Nx},\overline{Nhx})=\inf_{n\in N}\inf_{h\in H}d(x,nhx)=\inf_{g\in G}d(x,gx)$.

(2).\ The isomorphism is given by $\overline{G^\Omega r}\mapsto r'$, where, for $r\in X^\Omega$, we let $r'\in (X\sslash G)^\Omega$ be the random variable defined almost surely by $r'(\omega)=\overline{Gr(\omega)}$. Lemma~\ref{Borel selector}.(\ref{B.s. metric quotient}) ensures this map is surjective. We verify that it is isometric. Indeed, we have $$d^\Omega(r',s')=\mbE\inf_{g\in G}d(r,gs)=\inf_{g\in G^\Omega}d^\Omega(r,gs)=d^\Omega(\overline{G^\Omega r},\overline{G^\Omega s}),$$ where the second identity can be seen by approximating $r,s\in X^\Omega$ by random variables of finite range.

(3).\ Follows from (1) and (2).
\end{proof}

\begin{prop}\label{wreath on rando is app olig} Suppose $G$ is a Polish group acting approximately oligomorphically on a Polish metric space $(X,d)$. Then, the induced action $G\wr\Omega\actson (X^\Omega,d^\Omega)$ is also approximately oligomorphic.\end{prop}
\begin{proof}
Since $(X^\Omega)^n\simeq (X^n)^\Omega$, by the previous lemma we have $$(X^\Omega)^n\sslash (G\wr\Omega)\simeq K^\Omega\sslash\Aut(\Omega),$$ where the quotient $K=X^n\sslash G$ is assumed to be compact. As such, $K$ is a continuous image of the Cantor space, i.e., there exists a continuous surjective map $2^\mbN\to K$. This induces a natural continuous map $(2^\mbN)^\Omega\to K^\Omega$, which is surjective by Lemma~\ref{Borel selector}.(\ref{B.s. compact spaces}). Finally, this gives us a continuous surjective map $(2^\Omega)^\mbN\sslash\Aut(\Omega)\to K^\Omega\sslash\Aut(\Omega)$. Since the action $\Aut(\Omega)\actson 2^\Omega$ is approximately oligomorphic ($2^\Omega$ is a closed subset of $\mbA$), we deduce that $K^\Omega\sslash\Aut(\Omega)$ is compact, as we wanted. (In other words, $K^\Omega$ is an imaginary sort of the $\aleph_0$-categorical structure $\mbA$.)
\end{proof}

\begin{cor}\label{roelcke is preserved} If $G$ is a Polish Roelcke precompact group, then so is $G\wr\Omega$.\end{cor}
\begin{proof}
Follows from the previous proposition and Lemma~\ref{faithful actions}.
\end{proof}

It is true in general that the semidirect product of two Roelcke precompact groups is again Roelcke precompact (see \cite{tsaUnitary}, Proposition~2.2), but in our case $G^\Omega$ is not expected to be Roelcke precompact. For instance, if $G=2$ is the finite group with 2 elements, then $G^\Omega$ is non-compact and abelian, thus not Roelcke precompact. In fact, the group $G^\Omega$ alone may have rather unusual properties; see \cite{kailem}.

It will be useful to have an explicit description of the quotients $K^\Omega\sslash\Aut(\Omega)$ for compact~$K$: the orbit closures of compact-valued random variables should be seen as probability distributions.

\begin{lem}\label{quotients of randomized compact sets} Let $(K,d)$ be a compact metric space, and consider the induced action $\Aut(\Omega)\actson (K^\Omega,d^\Omega)$. Then, the quotient $K^\Omega\sslash\Aut(\Omega)$ is homeomorphic to the space $\Reg(K)$ of Borel probability measures on $K$.
\end{lem}
\begin{proof}
Let $\pi\colon K^\Omega\to K^\Omega\sslash\Aut(\Omega)$ be the quotient map. Given $r\in K^\Omega$, the pushforward of the Lebesgue measure by $r$ is the measure $r_*\mu\in\Reg(K)$ defined by $\int_K fdr_*\mu=\int_\Omega frd\mu$. We consider the map $\theta\colon\pi(r)\mapsto r_*\mu$, which is clearly well-defined and continuous.

Suppose $r_*\mu=s_*\mu$. Then, given any finite algebra $B$ of Borel subsets of $K$, the preimages $r^{-1}(B)$ and $s^{-1}(B)$ are isomorphic measure algebras. Hence, by the homogeneity of MALG$_\mu$, there is $t\in\Aut(\Omega)$ such that $t^{-1}(r^{-1}(B))=s^{-1}(B)$. By duality, this yields $\pi(r)=\pi(s)$, so $\theta$ is injective. Conversely, given any measure $\nu\in\Reg(K)$, the associated measure algebra MALG$(K,\nu)$ is separable, thus it embeds into the measure algebra of $\Omega$. By duality, this induces a measure preserving transformation $r_\nu\colon\Omega\to (K,\nu)$, that is, $(r_\nu)_*\mu=\nu$. Hence, $\theta$ is a continuous bijection. Since $K^\Omega\sslash\Aut(\Omega)$ is compact, it is moreover a homeomorphism.
\end{proof}

Before passing to the next section we record the following expected counterpoint to the previous facts.

\begin{lem}\label{G times AutOmega on XOmega}
Let $(X,d)$ be a Polish metric space and let $G\leq\Iso(X,d)$ be any Polish subgroup of isometries of $X$. We may see $G$ as the subgroup of constant elements of $G^\Omega$. Then, the quotient $X^\Omega\sslash(G\times\Aut(\Omega))$ is compact if and only if $X$ is compact.\end{lem}
\begin{proof}
We already know one implication. For the converse, if $X$ is not compact, let $x_i\in X$, $i\in\mbN$, be such that $d(x_i,x_j)>\epsilon$ for $i\neq j$. For $n\in\mbN$, let $\{A^n_i\}_{i<2^n}$ be a partition of $\Omega$ by sets of measure $1/2^n$. Take $r_n=\sum_{i<n}x_i\chi_{A^n_i}$. We claim that the sequence $r_n$ has no convergent subsequence in $X^\Omega\sslash G\times\Aut(\Omega)$.

Suppose for a contradiction that there are a subsequence $\tilde{r}_n=r_{m(n)}$, $\tilde{A}^n_i=A^{m(n)}_i$, and elements $g_n\in G$, $t_n\in\Aut(\Omega)$ such that $g_nt_n\tilde{r}_n=\sum_{i<m(n)}g_nx_i\chi_{t_n\tilde{A}^n_i}$ converges in $X^\Omega$ to a random variable $r$. Let $\delta=\epsilon/16$. For some $N$ we have $\int d(r,\tilde{r}_n)d\mu<\delta$ whenever $n\geq N$. By Chebyshev's inequality, the set $A\subset\Omega$ where we have simultaneously $d(r,\tilde{r}_N)\leq 4\delta$ and $d(r,\tilde{r}_{N+1})\leq 4\delta$ has measure $\mu(A)>1/2$. We can deduce that there are $i,j,k$, $j\neq k$, such that both $A\cap \tilde{A}^N_i\cap \tilde{A}^{N+1}_j$ and $A\cap \tilde{A}^N_i\cap \tilde{A}^{N+1}_k$ are non-empty; say $\omega_j$ is in the former intersection and $\omega_k$ in the latter. Then
\begin{align*}
\epsilon & <d(g_{N+1}x_j,g_{N+1}x_k)=d(\tilde{r}_{N+1}(\omega_j),\tilde{r}_{N+1}(\omega_k))\leq 8\delta+d(r(\omega_j),r(\omega_k))\\
& \leq 16\delta+d(\tilde{r}_N(\omega_j),\tilde{r}_N(\omega_k))=\epsilon+d(g_Nx_i,g_Nx_i)=\epsilon,
\end{align*}
a contradiction.
\end{proof}

\subsection{The automorphism group of the Borel randomization}\label{ss:autMR}

In this subsection we fix a (separable, metric) logic structure $M$, with at least two elements. The automorphism group of $M$, $\Aut(M)$, is a Polish subgroup of the group of isometries of $M$ (that is, with the topology of pointwise convergence). For simplicity of notation we will denote $G=\Aut(M)$.

Let $M^R$ be the Borel randomization of $M$, and let $G^R=\Aut(M^R)$. The groups $G^\Omega$ and $\Aut(\Omega)$ act faithfully on $M^\Omega$, the main sort of $M^R$. Additionally, $\Aut(\Omega)$ acts on the auxiliary sort $\mbA$, and we can consider the trivial action of $G^\Omega$ on $\mbA$, i.e., each $g\in G^\Omega$ acts as the identity of $\mbA$. Combining the actions on each sort, we obtain actions $G^\Omega\actson M^R$, $\Aut(\Omega)\actson M^R$, which are clearly by isomorphisms. Using Lemma~\ref{faithful actions}, we deduce that $G\wr\Omega$ is a topological subgroup of $G^R$.

\begin{lem}\label{identity on the auxiliary sort} If $g\in G^R$ is the identity on the auxiliary sort, then $g\in G^\Omega$.\end{lem}
\begin{proof}
If $g$ is the identity on the auxiliary sort, then for every $L$-formula $\varphi$ and random variable $r\in M^\Omega$ we have $\llbracket\varphi(gr)\rrbracket=g\llbracket\varphi(r)\rrbracket=\llbracket\varphi(r)\rrbracket$, that is, \begin{equation}\label{fgr=fr}\varphi((gr)(\omega))=\varphi(r(\omega))\text{ for almost every }\omega.\end{equation}

Let $D\subset M$ be a countable dense subset and consider $C\subset M^R$ the family of constant random variables taking a value from $D$. Since the language $L$ is countable, by (\ref{fgr=fr}) we have $$\varphi((gc)(\omega))=\varphi(c)$$ for every formula $\varphi$, every tuple $c$ from $C$ and every $\omega$ in a common full-measure set $F\subset\Omega$. Now, for $\omega\in F$ and $c\in D$ we define $$g(\omega)(c)\coloneqq (gc)(\omega)$$ (where the $c$ on the right is the corresponding constant function on $C$), and this induces and elementary map $g(\omega)\colon D\to M$, which extends by continuity to an endomorphism of $M$.

Next we check that, for every $r\in M^\Omega$, \begin{equation}\label{grw=gwrw}(gr)(\omega)=g(\omega)(r(\omega))\text{ for almost every }\omega.\end{equation}
By (\ref{fgr=fr}), for each $r\in M^\Omega$ there is a full-measure subset $F'\subset F$ such that $$d(gr(\omega),gc(\omega))=d(r(\omega),c)$$ for every $c\in C$ and $\omega\in F'$. Let $\omega\in F'$, and take a sequence $c_n\in C$ such that $c_n\to r(\omega)$. Then we have $d(gr(\omega),gc_n(\omega))=d(r(\omega),c_n)\to 0$, and on the other hand $gc_n(\omega)=g(\omega)(c_n)\to g(\omega)(r(\omega))$. That is to say, $gr(\omega)=g(\omega)(r(\omega))$. This proves (\ref{grw=gwrw}).

Note that $g(\omega)$ is surjective (i.e., $g(\omega)\in G$) for almost every $\omega$. Indeed, since the image of $g(\omega)$ is closed, it is enough to see that it contains $D$. But for a constant $c\in C$ we have, by~(\ref{grw=gwrw}), $g(\omega)(g^{-1}c(\omega))=g(g^{-1}c(\omega))=c(\omega)=c$ for every $\omega$ in a full-measure set that depends on $c$. Since $D$ is countable, we are done.

Finally, we check that the map $\omega\in F\subset\Omega\mapsto g(\omega)\in G$ is measurable, which shows that it belongs to $G^\Omega$. It is enough to see that, for every $c,d\in D$ and $\epsilon>0$, the set $A=\{\omega\in F:d(g(\omega)(c),d)<\epsilon\}$ is measurable. This is clear, since $A=\{\omega\in F:d(gc(\omega),d)<\epsilon\}$ and $gc\in M^\Omega$ is measurable.
\end{proof}

Let $\Aut(\mbA)$ be the automorphism group of $\mbA$ as a reduct of $M^R$. It is easy to check that the map $t\in\Aut(\Omega)\mapsto t^*\in\Aut(\mbA)$ is surjective, where $(t^*R)(\omega)=R(t(\omega))$ for $R\in\mbA$. (That is, the left actions $\Aut(\Omega)\actson\mbA$ and $\Aut(\mbA)\actson\mbA$ are anti-isomorphic; remark also that, as topological groups, $\Aut(\Omega)\simeq\Aut(\mbA)$, since every group is anti-isomorphic to itself.)

\begin{theorem}\label{thm G^R}
For every separable structure $M$ we have $\Aut(M^R)=\Aut(M)\wr\Omega$.
\end{theorem}
\begin{proof}
We have already established that $G\wr\Omega$ is a topological subgroup of $G^R$. Let $\sigma\in G^R$. The restriction of $\sigma$ to the auxiliary sort induces an automorphism of $\mbA$, say $t^*$ for $t\in\Aut(\Omega)$. Let $g=\sigma t$. Then $g$ is the identity on the auxiliary sort, so by the previous lemma we have $g\in G^\Omega$. We conclude that $G^R$ is the product of $G^\Omega$ and $\Aut(\Omega)$, so the proof is complete.
\end{proof}

As an application we get a new proof of the following preservation result of \cite{benkei,benOntheories}. Recall that $M$ is \emph{$\aleph_0$-categorical} if every separable model of its first-order theory is isomorphic to $M$.

\begin{cor}\label{separable categoricity is preserved} If $M$ is $\aleph_0$-categorical, then so is $M^R$.\end{cor}
\begin{proof}
By the continuous version of the theorem of Ryll-Nardzewski, a structure $N$ is $\aleph_0$-categorical if and only if the action $\Aut(N)\actson N$ is approximately oligomorphic (see \cite[\textsection 5]{bentsa}). Thus, the result follows from the previous theorem and Proposition~\ref{wreath on rando is app olig}.
\end{proof}

In addition to the automorphism group of $M$, one can consider the topological semigroup of endomorphisms (elementary self-embeddings) of $M$, which we denote by $\End(M)$. For $\aleph_0$-categorical structures we have the following pleasant fact, observed in \cite[\textsection 2.2]{bentsa}.

\begin{prop}\label{prop G_L=End} Suppose $M$ is $\aleph_0$-categorical, $G=\Aut(M)$. Then $\widehat{G}_L=\End(M)$, that is, $\End(M)$ is the left completion of $\Aut(M)$.\end{prop}
\begin{proof}
See \cite{bit15}, Fact~2.14. The proof adapts readily to the case of metric structures, as per \cite{bentsa}, Lemma~2.3.
\end{proof}

For instance, the left completion of $\Aut(\Omega)\simeq\Aut(\mbA)$ is $\End(\mbA)$, and the latter is anti-isomorphic to $\End(\Omega)$ by the map $t\in\End(\Omega)\mapsto t^*\in\End(\mbA)$. This shows that $\End(\Omega)$ is the right completion of $\Aut(\Omega)$.

\begin{cor}\label{End(M^R)} Let $M$ be an $\aleph_0$-categorical structure, $G=\Aut(M)$. Then $\End(M^R)=\widehat{(G^R)}_L=\End(M)^\Omega\rtimes\End(\Omega)^*$.\end{cor}
\begin{proof}
Combine Corollary~\ref{separable categoricity is preserved}, Theorem~\ref{thm G^R}, Proposition~\ref{prop G_L=End} and Lemma~\ref{left completion of GwrOmega}.
\end{proof}

\begin{rem}\label{rem submodels of M^R} Suppose $M$ is $\aleph_0$-categorical. Since $M^R$ is $\aleph_0$-categorical too, the elementary submodels of $M^R$ are the images of its elementary self-embeddings. We deduce from the previous corollary that (the main sort of) a submodel of $M^R$ consists of random variables of the form $hs^*r$ for $r\in M^\Omega$ and some fixed pair $h\in\End(M)^\Omega$, $s\in\End(\Omega)$. In other words, a submodel is given by choosing (measurably) for each $\omega\in\Omega$ a submodel $M_\omega\prec M$ (the image of $h(\omega)$), then considering all sections $\Omega\to\bigcup_{\omega\in\Omega}M_\omega$ that are measurable with respect to a fixed factor of $\Omega$ (the $\sigma$-algebra generated by $s$).\end{rem}

\noindent\hrulefill

\section{Randomized compactifications}\label{s:rando of ambits}

\subsection{The Markov randomization}\label{ss:Markov}

The results of the previous section support the view that the randomization of an isometric action $G\actson (X,d)$ should be considered to be the action $G\wr\Omega\actson (X^\Omega,d^\Omega)$. A natural question is what should be considered as the randomization of a $G$-flow, that is, of a continuous action $G\actson K$ on a compact space. Seemingly, there is not a canonical answer for this question. In this subsection, we will introduce a construction that provides a satisfactory answer for the \emph{Roelcke compactification} of Roelcke precompact Polish groups.

Given a compact metrizable space $K$ and a probability measure $\lambda\in\Reg(\Omega\times K)$, we will denote by $\lambda|_\Omega$ the pushforward of $\lambda$ by the projection $\Omega\times K\to\Omega$. In what follows, we fix two copies $\Omega_0$ and $\Omega_1$ of $\Omega$; we still denote the Lebesgue measure on each of them by $\mu$.

\begin{defin} Let $K$ be a compact metrizable space. We define the \emph{Markov randomization} of $K$ as the compact space $$\mcM(\Omega,K)\coloneqq\{\lambda\in\Reg(\Omega_0\times K\times\Omega_1):\lambda|_{\Omega_0}=\mu,\ \lambda|_{\Omega_1}=\mu\}.$$
\end{defin}

For instance, if $1$ denotes the one-point space, then $\mcM(\Omega,1)$ is just the space of self-joinings of the Lebesgue measure, which can be identified with the space of Markov operators of $L^2(\Omega)$ (see, for instance, \cite[Ch.\ 6, \textsection 2]{glasnerBook}).

If $G\actson K$ is a continuous action of a Polish group on $K$, then we have an induced continuous action $G\wr\Omega\actson\mcM(\Omega,K)$. In order to describe it, we observe first that we can identify $$\mcM(\Omega,K)\simeq E\sslash\Aut(\Omega),$$ where $E=\{(s,k,r)\in (\Omega_0\times K\times\Omega_1)^\Omega:s_*\mu=\mu,\ r_*\mu=\mu\}$, and $\Aut(\Omega)$ acts on it by restriction of the action $\Aut(\Omega)\actson (\Omega_0\times K\times\Omega_1)^\Omega$. Indeed, the identification follows from a straightforward adaptation of the proof of Lemma~\ref{quotients of randomized compact sets}; the measure $\lambda$ corresponding to the class of a triple $(r,k,s)$ is defined by the relation $\int fd\lambda=\int f(r,k,s)d\omega$, for $f\in C(\Omega_0\times K\times\Omega_1)$. Next, we remark that $E$ is naturally homeomorphic to the product $\End(\Omega)\times K^\Omega\times\End(\Omega)$. The corresponding action of $\Aut(\Omega)$ is given by $t\cdot(s,k,r)=(st^{-1},kt^{-1},rt^{-1})$, for $t\in\Aut(\Omega)$, $s,r\in\End(\Omega)$ and $k\in K^\Omega$. Hence, by considering the quotient with respect to this action, we have $$\mcM(\Omega,K)\simeq\left(\End(\Omega)\times K^\Omega\times\End(\Omega)\right)\sslash\Aut(\Omega).$$ Let us denote by $[s,k,r]\in\mcM(\Omega,K)$ the image by the previous homeomorphism of the class of the triple $(s,k,r)$. Then, given $t\in\Aut(\Omega)$, $g\in G^\Omega$, we have natural actions $$t[s,k,r]\coloneqq [ts,k,r],\ g[s,k,r]\coloneqq [s,(s^*\cdot g)k,r].$$ Here, $(s^*\cdot g) k(\omega)=g(s(\omega))(k(\omega))$. These actions are compatible with the multiplication law of the semidirect product of $G^\Omega$ and $\Aut(\Omega)$, thus they induce an action $G\wr\Omega\actson\mcM(\Omega,K)$, which is moreover continuous.

We recall that a \emph{compactification} (or \emph{ambit}) of $G$ is given by a continuous action $G\actson K$ on a compact Hausdorff space together with a distinguished point $k_0\in K$ such that $Gk_0$ is dense in $K$.

\begin{prop} If $(K,k_0)$ is a metrizable compactification of $G$, then $\mcM(\Omega,K)$ is a (metrizable) compactification of $G\wr\Omega$ with distinguished point $[1,k_0,1]$ (where $k_0\in K^\Omega$ is seen as a constant random variable).\end{prop}
\begin{proof}
We have to check that the orbit of $[1,k_0,1]$ by $G\wr\Omega$ is dense in $\mcM(\Omega,K)$. Since $Gk_0$ is dense in $K$, we can approximate the elements of $K^\Omega$ by random variables of finite range taking values in $Gk_0$; that is, $G^\Omega k_0$ is dense in $K^\Omega$. Now, if we are given $[s,k,r]\in\mcM(\Omega,K)$, we can take $t_n,u_n\in\Aut(\Omega)$, $g_n\in G^\Omega$, such that $t_n\to s$, $u_n\to r$ and $g_nk_0\to k$. Then, $[t_n,g_nk_0,u_n]=t_ng_nu_n^{-1}[1,k_0,1]$ is in the orbit of $[1,k_0,1]$ and converges to $[s,k,r]$.
\end{proof}

Let $G$ be a Polish group, $\widehat{G}_L$ its left completion. The completion of $G$ with respect to its \emph{lower uniformity} (the infimum of the left and right uniformities) is called the \emph{Roelcke completion} of~$G$. As observed in \cite[\textsection 2.1]{bentsa}, the Roelcke completion of~$G$ can be identified with the metric quotient $R(G)=(\widehat{G}_L\times\widehat{G}_L)\sslash G$. Note that $R(G)$ is a Polish space, and $G$ acts continuously on it by the formula $$g\overline{G(x,y)}=\overline{G(xg^{-1},y)},$$ where $g\in G$ and $x,y\in\widehat{G}_L$. Theorem~2.4 in \cite{bentsa} shows that $G$ admits a faithful approximately oligomorphic action if and only if $R(G)$ is compact (i.e., $G$ is \emph{Roelcke precompact}). In that case, $R(G)$ is a metrizable compactification of $G$ (with distinguished point $\overline{G(1,1)}$), called the \emph{Roelcke compactification} of $G$.

\begin{theorem}
Let $G$ be a Polish Roelcke precompact group. Then we have an isomorphism of compactifications of $G\wr\Omega$, $$R(G\wr\Omega)\simeq\mcM(\Omega,R(G))$$ (i.e., a $G\wr\Omega$-equivariant homeomorphism respecting the distinguished points).
\end{theorem}
\begin{proof}
Using Corollary~\ref{left completion of GwrOmega} and Lemma~\ref{algebra of quotients},
\begin{align*}
R(G\wr\Omega) & \simeq\left((\widehat{G}_L)^\Omega\times\End(\Omega)^*\times(\widehat{G}_L)^\Omega\times\End(\Omega)^*\right)\sslash (G\wr\Omega)\\
              & \simeq\left(\End(\Omega)^*\times\left((\widehat{G}_L\times\widehat{G}_L)^\Omega\sslash G^\Omega\right)\times\End(\Omega)^*\right)\sslash\Aut(\Omega)\\
              & \simeq\left(\End(\Omega)\times R(G)^\Omega\times\End(\Omega)\right)\sslash\Aut(\Omega)\\
              & \simeq\mcM(\Omega,R(G)).
\end{align*}
It is easy to verify that the given homeomorphism respects the actions of $G\wr\Omega$ and the distinguished points.
\end{proof}

We remark next that the Markov randomization behaves well with respect to \emph{semitopological semigroups}. Recall that a topological space with a semigroup law is semitopological if multiplication is separately continuous.

Let $S$ be a compact metrizable semitopological semigroup. Then, the space $S^\Omega$ is a semitopological semigroup with pointwise multiplication. Notice that, since $S$ is separable metrizable, then the product in $S$, being separately continuous, is in fact jointly measurable. Hence the pointwise product of two elements of $S^\Omega$ is again in $S^\Omega$.

Now we can define a product on $\mcM(\Omega,S)$, as follows. Given $\lambda,\nu\in\mcM(\Omega,S)$, it is always possible to find $t,s,r\in\End(\Omega)$ and $p,q\in S^\Omega$ such that $\lambda=[t,p,s]$, $\nu=[s,q,r]$, and such that the $\sigma$-algebra on $\Omega$ generated by $t,p$ is relatively independent from the $\sigma$-algebra generated by $q,r$ over the $\sigma$-algebra generated by $s$. Then, we set $$\lambda\nu\coloneqq [t,pq,r].$$ The relative independence condition ensures the good definition. Alternatively, the measure $\lambda\nu$ can be defined by the formula $$\int fd\lambda\nu=\iiint f(\omega_0,xy,\omega_1)d\lambda^\omega(\omega_0,x)d\nu_\omega(y,\omega_1)d\omega$$ for $f\in C(\Omega_0\times S\times\Omega_1)$, where $\lambda^\omega$, $\nu_\omega$ are given by the disintegrations of $\lambda$, $\nu$ over $\Omega_1$ and $\Omega_0$, respectively, i.e., $$\lambda=\int\lambda^{\omega_1}\times\delta_{\omega_1}d\omega_1,\ \nu=\int\delta_{\omega_0}\times\nu_{\omega_0}d\omega_0.$$

The product thus defined is associative and separately continuous; we omit the (routine) verification. Hence we have the following.

\begin{prop}
If $S$ is a compact, metrizable, semitopological semigroup, then so is $\mcM(\Omega,S)$, with the product defined above.
\end{prop}

Moreover, if $S$ is a compactification of $G$ and the semigroup law of $S$ is compatible with the action of $G$ (i.e., if we have $gs=(gs_0)s$ for every $g\in G$, $s\in S$, where $s_0$ is the distinguished point of the compactification), then the law of $\mcM(\Omega,S)$ is compatible with the action of $G\wr\Omega$. In particular, if $R(G)$ is a semitopological semigroup compactification of $G$ (that is, if it admits a semitopological semigroup law compatible with the group law of $G$), then $R(G\wr\Omega)$ is a semitopological semigroup compactification of $G\wr\Omega$. Suppose that $G=\Aut(M)$ for an $\aleph_0$-categorical structure $M$. It follows from \cite{bentsa}, Theorem~5.5, that $M$ is \emph{stable} if and only if $R(G)$ is a semitopological semigroup compactification. Hence, from these two facts together we get a new proof (in the $\aleph_0$-categorical case) of the preservation of stability by randomizations: if $T$ is stable, then the randomized theory $T^R$ is stable (\cite[\textsection 4.2]{benOntheories}\cite[\textsection 5.3]{benkei}).

We can also recover the formula-by-formula version of this result. For the basic definitions of stability theory in the metric setting, see \cite[\textsection 7]{benusv10}; for us, however, it is more convenient to work with the dynamical translation of \cite{bentsa}. Let $W(G)$ be the \emph{WAP compactification} of $G$, which is the largest semitopological semigroup compactification of $G$. As per \cite[\textsection 5]{bentsa}, every \emph{stable formula} induces a function in $C(W(G))$ and, conversely, every function $\varphi\in C(W(G))$ is induced by a formula $\varphi(x,y)$ on the $\aleph_0$-categorical structure~$M$ (defined on a certain domain) which is moreover stable.

Now, given a continuous function $\varphi\in C(K)$, we can define an associated function $\mbE\llbracket\varphi\rrbracket\in C(\mcM(\Omega,K))$, given by $$\mbE\llbracket\varphi\rrbracket([s,k,r])=\int \varphi(k(\omega))d\omega.$$ We take $K=W(G)$. By our previous propositions, $\mcM(\Omega,W(G))$ is a semitopological semigroup compactification of $G\wr\Omega$, hence a factor of the largest one such, $W(G\wr\Omega)$. Thus, if $\varphi$ is a continuous function on the WAP compactification of $G$, then $\mbE\llbracket\varphi\rrbracket$ factors through the WAP compactification of $G\wr\Omega$.

Under the translation between formulas and functions, if $\varphi$ corresponds to a formula $\varphi(x,y)$, then the function $\mbE\llbracket\varphi\rrbracket$ corresponds to the formula $\mbE\llbracket\varphi(x,y)\rrbracket$. Hence, by the previous discussion, we recover (for $\aleph_0$-categorical theories) the strong form of the preservation of stability (\cite{benOntheories}, Theorem~4.9).

\begin{cor}\label{cor pres stability}
If $\varphi(x,y)$ is stable for $T$, then $\mbE\llbracket\varphi(x,y)\rrbracket$ is stable for $T^R$.
\end{cor}

The na\"ive converse of the previous fact is obvious: if $\varphi(x,y)$ is unstable, then $\mbE\llbracket\varphi(x,y)\rrbracket$ is unstable (with the order property witnessed even by constant random variables). However, one may ask for a more subtle converse: is every \emph{stable} formula $\Phi(x,y)$ in $T^R$ (say, with variables from the main sort) equivalent to a continuous combination of formulas of the form $\mbE\llbracket\varphi(x,y)\rrbracket$ for \emph{stable} formulas $\varphi(x,y)$? We prefer to pose the question in the following terms.

\begin{question} Do we have $W(G\wr\Omega)\simeq\mcM(\Omega,W(G))$ for every Roelcke precompact Polish group $G$?\end{question}

\begin{rem} The \emph{Bohr compactification} of $G\wr\Omega$, that is, the largest topological group compactification of $G\wr\Omega$, is trivial (i.e., a singleton). Indeed, the Bohr compactification of the automorphism group of an $\aleph_0$-categorical structure $N$ can be identified with the automorphism group of $\acl^N(\emptyset)$, the (imaginary) algebraic closure of the empty set in~$N$, as follows from \cite{benBohr} (see also \cite[\textsection 1.5]{iba14}). However, the algebraic closure of $\emptyset$ in $M^R$ is trivial (regardless of $M$), in the sense that it coincides with the definable closure of $\emptyset$, as follows from \cite{benOntheories}, Theorem~5.9.\end{rem}

\subsection{Hilbert-representability}

As mentioned above, if an $\aleph_0$-categorical structure $M$ is stable, then the Roelcke completion of its automorphism group, $R(G)$, is a compact semitopological semigroup, and vice versa. In that case, by a general result of Shtern \cite{shtern94}, $R(G)$ can be embedded (topologically and homomorphically) into the compact semitopological semigroup $$\Theta(V)=\{T\in L(V):\|T\|\leq 1\}$$ of linear contractions of a \emph{reflexive} Banach space $V$ (endowed with the weak operator topology). Thus, an interesting stronger property is satisfied if the space $V$ can be chosen to be a Hilbert space.

\begin{defin} A semitopological semigroup $S$ is \emph{Hilbert-representable} if it can be embedded into $\Theta(\mcH)$ for a Hilbert space $\mcH$.\end{defin}

For the case of $R(G)$, this property is therefore a strengthening of stability, and has been investigated as such in \cite{bit15}. We showed there that, for a \emph{classical} $\aleph_0$-categorical structure $M$, $R(G)$ is a Hilbert-representable semitopological semigroup if and only if $M$ is stable and one-based (equivalently, $\aleph_0$-stable). It is unclear how to generalize this for metric structures; we will come back to this discussion in Section~\ref{s:beautiful}. Here, we show that this property is preserved under randomizations.

Given a Hilbert space $\mcH$, we denote by $\mcH^{\otimes n}$ the $n$-fold tensor product $\mcH\otimes\dots\otimes\mcH$ of Hilbert spaces. Also, we write $$\mcH^\otimes\coloneqq\bigoplus_{n\in\mbN}\mcH^{\otimes n}$$ for the direct sum of all the $n$-fold tensor self-products of $\mcH$. We recall that every linear contraction of $\mcH$ acts naturally as a linear contraction on each $\mcH^{\otimes n}$ (satisfying the identity $T(u_1\otimes\dots\otimes u_n)=Tu_1\otimes\dots\otimes Tu_n$), and hence also on the direct sum $\mcH^\otimes$. That is, we have an inclusion of semitopological semigroups, $\Theta(\mcH)<\Theta(\mcH^\otimes)$.

\begin{theorem}
Let $S$ be a compact metrizable semitopological semigroup. If $S$ is Hilbert-representable, then so is $\mcM(\Omega,S)$.
\end{theorem}
\begin{proof}
Let $\beta\colon S\to\Theta(\mcH)$ be an embedding into the semigroup of contractions of a Hilbert space $\mcH$, which we can also see as an embedding $\beta\colon S\to\Theta(\mcH^\otimes)$. We consider the map $$\beta^R\colon\mcM(\Omega,S)\to\Theta(L^2(\Omega,\mcH^{\otimes}))$$ defined implicitly by the inner product $$\langle f_0,\beta^R(\lambda)f_1\rangle=\int\langle f_0(\omega_0),\beta(x)f_1(\omega_1)\rangle d\lambda(\omega_0,x,\omega_1),$$ where $\lambda\in\mcM(\Omega,S)$ and $f_0,f_1\in L^2(\Omega,\mcH^\otimes)$. It is checked easily that $\beta^R(\lambda)$ is a linear contraction of $L^2(\Omega,\mcH^{\otimes})$ for every $\lambda$. Besides, if the functions $f_0,f_1\colon\Omega\to\mcH^\otimes$ are continuous, then the function $(\omega_0,x,\omega_1)\mapsto\langle f_0(\omega_0),\beta(x)f_1(\omega_1)\rangle$ is continuous; hence, if $\lambda_n$ converge to~$\lambda$, then the integrals $\langle f_0,\beta^R(\lambda_n)f_1\rangle$ converge to $\langle f_0,\beta^R(\lambda)f_1\rangle$. If $f_0,f_1$ are not continuous, we can approximate them in norm by continuous functions $f_0',f_1'$; in particular, the inner product $\langle f_0,\beta^R(\lambda)f_1\rangle$ is approximated by $\langle f_0',\beta^R(\lambda)f_1'\rangle$, uniformly on $\lambda$. We see, thence, that $\beta^R$ is continuous.

We now check that $\beta^R$ is a homomorphism. Let $\lambda,\nu\in\mcM(\Omega,S)$. We note that, for almost every $\omega$, $\beta^R(\nu)f_1(\omega)$ equals the vector-valued integral $\int\beta(y)f_1(\omega_1)d\nu_{\omega}(y,\omega_1)$. Hence,
\begin{align*}
\langle f_0,\beta^R(\lambda)\beta^R(\nu)f_1\rangle & =\iint\langle f_0(\omega_0),\beta(x)(\beta^R(\nu)f_1(\omega))\rangle d\lambda^\omega(\omega_0,x)d\omega \\
& =\iiint\langle f_0(\omega_0),\beta(x)\beta(y)f_1(\omega_1))\rangle d\nu_\omega(y,\omega_1)d\lambda^\omega(\omega_0,x)d\omega \\
& =\langle f_0,\beta^R(\lambda\nu)f_1\rangle.
\end{align*}

Since $\mcM(\Omega,S)$ is compact, we are only left to show that $\beta^R$ is injective. Since $\beta$ is an embedding, the continuous functions $x\mapsto\langle u_0,\beta(x)u_1\rangle$ for $u_0,u_1\in\mcH$ separate points of~$S$. Hence, by the Stone--Weierstrass theorem, the unital algebra $A$ generated by them is dense in $C(S)$. The key observation then is that a function $h\in A$ is always of the form $h(x)=\langle w_0,\beta(x)w_1\rangle$ for appropriate vectors $w_0,w_1\in\mcH^\otimes$. Now, if $h$ is one such function and we are given $e_0,e_1\in C(\Omega)$, we consider $f_0,f_1\in L^2(\Omega,\mcH^\otimes)$ defined by $f_0(\omega_0)=e_0(\omega_0)w_0$, $f_1(\omega_1)=e_1(\omega_1)w_1$. Then, if $\beta^R(\lambda)=\beta^R(\nu)$, the identity $\langle f_0,\beta^R(\lambda)f_1\rangle=\langle f_0,\beta^R(\nu)f_1\rangle$ becomes $$\int e_0he_1d\lambda=\int e_0he_1d\nu.$$ Since this holds for arbitrary $e_0,e_1\in C(\Omega)$, $h\in A$, it follows that $\lambda=\nu$.
\end{proof}

As a particular case we obtain the following.

\begin{cor}
If $G=\Aut(M)$ for a classical $\aleph_0$-categorical $\aleph_0$-stable structure $M$, then $R(G\wr\Omega)$ is a Hilbert-representable semitopological semigroup.
\end{cor}
\begin{proof}
Follows from the previous results together with \cite{bit15}, Corollary~3.12.
\end{proof}

For any topological group $G$, there is always a largest Hilbert-representable semitopological semigroup compactification of $G$, which we denote by $H(G)$.

\begin{question} Is it $H(G\wr\Omega)\simeq\mcM(\Omega,H(G))$ for every Roelcke precompact Polish group~$G$?\end{question}

\subsection{A general preservation result}\label{ss:a general pres res} We end this section with a general preservation result about Banach representations of randomized type spaces (Theorem~\ref{thm general preservation} below). Modulo some additional theory, this is indeed a generalization of the preservation results discussed above, concerning stability and Hilbert-representability. Moreover, it allows us to recover, in the $\aleph_0$-categorical case, the main result of \cite{benVapChe} of preservation of NIP formulas and theories.

In the previous subsections, we have considered a particular way of randomizing $G$-flows or, more generally, some interesting compact spaces. From a model-theoretic point of view, the main compact spaces associated with a structure $M$ are its type spaces. In particular, the space $S(M)$ of complete types with parameters from $M$ captures a large amount of model-theoretic information about the structure (and, in some cases, even about its theory). With the natural map $M\to S(M)$, this type space is a compactification of $M$ in which $M$ embeds. Thus, the natural randomized object to consider in this context is the compactification $M^R\to S(M^R)$, which we describe below.

Given a compact metrizable space $K$, we will consider the subspace of $L^2_{w^*}(\Omega,C(K)^*)$ (with the weak$^*$ topology) consisting of those elements $p$ that take values in the Borel probability measures on $K$. We observe that we have a homeomorphism $$\{p\in L^2_{w^*}(\Omega,C(K)^*):p(\omega)\in\Reg(K)\ \mu\text{-a.e.}\}\simeq\{\lambda\in\Reg(\Omega\times K):\lambda|_\Omega=\mu\},$$ where each $p$ corresponds to the measure $\lambda$ that can be disintegrated as $\lambda_\omega=p(\omega)$ almost everywhere. For convenience, we introduce a notation for this space.

\begin{defin} For a compact metrizable space $K$ we define $$\mcS(\Omega,K)\coloneqq\{\lambda\in\Reg(\Omega\times K):\lambda|_\Omega=\mu\},$$ which we may identify with $\{p\in L^2_{w^*}(\Omega,C(K)^*):p(\omega)\in\Reg(K)\ \mu\text{-a.e.}\}$ when convenient.\end{defin}

If $G\actson K$ is a continuous action, then we have an induced continuous action ${G\wr\Omega}\actson\mcS(\Omega,K)$. Indeed, observe first that $G$ acts continuously by isometries on $C(K)$, by $(gf)(x)=f(g^{-1}x)$ for $f\in C(K)$, $g\in G$ and $x\in K$. Hence we have an induced action $G\wr\Omega\actson L^2_{w^*}(\Omega,C(K)^*)$ (as per \textsection\ref{ss:Bochner}), which restricts to a continuous action on $\mcS(\Omega,K)$.

Now, fix any separable metric structure $M$ in a countable language $L$. Given a set $\Delta$ of $L$-formulas $\varphi(x,y)$, we let $S_\Delta(M)$ be the space of quantifier-free $\Delta$-types in the variable $x$ with parameters from $M$, which is a compact metrizable space. The value of a type $q\in S_x(M)$ on a formula $\varphi(x,b)$ is denoted by $\varphi(x,b)^q$ (this is a real number in $[0,1]$). If $G=\Aut(M)$, then $G$ acts continuously on $S_\Delta(M)$ by the relation $\varphi(x,b)^{gp}=\varphi(x,g^{-1}b)^p$.

In addition, we let $\Delta^R$ be the set of $L^R$-formulas of the form $\mbE\tau(x,y,z)$ where $x,y$ are tuples of variables from the main sort, $z$ is a tuple of variables from the auxiliary sort, and $\tau(x,y,z)$ is a term on $xyz$ built upon the formulas of $\Delta$. More precisely, $\tau(x,y,z)$ is constructed by applying any operations of the auxiliary sort to any variables from $z$ and to any term of the form $\llbracket\varphi(x,y)\rrbracket$ for $\varphi\in\Delta$. Then, we can consider the space $S_{\Delta^R}(M^R)$ of quantifier-free $\Delta^R$-types in the variable~$x$ (thus, from the main sort) with parameters from $M^R$, and the corresponding action $G\wr\Omega\actson S_{\Delta^R}(M^R)$.

\begin{rem} Let $\tau(x,y,z)$ be a term as above. Then, if we substitute $y$ by a tuple $b$ from $M$, and we substitute $z$ by a tuple of real numbers $c$, then $\tau(x,b,c)$ can be interpreted naturally as an $L$-formula with parameters from $b$, which is moreover obtained by a combination of formulas $\varphi(x,b)$ for $\varphi(x,y)\in\Delta$. In particular, for $q\in S_\Delta(M)$ the value $\tau(x,b,c)^q$ is defined, and this induces a continuous function $\tau(x,b,c)\colon S_\Delta(M)\to [0,1]$.\end{rem}

In the rest of the paper, for simplicity of notation, given $f\in C(K)$ and $\nu\in\Reg(K)$, we may denote the expected value $\int fd\nu$ by $\mbE^\nu(f)$.

\begin{lem}\label{lem types as random variables}
For any set of $L$-formulas $\Delta$, we have a $G\wr\Omega$-equivariant homeomorphism $S_{\Delta^R}(M^R)\simeq\mcS(\Omega,S_\Delta(M))$.

Under this identification, a type $p\in S_\Delta(M^R)$ can be seen as a random variable with values in $\Reg(S_\Delta(M))$, and such that $$\mbE\tau(x,r,s)^p=\int\mbE^{p(\omega)}(\tau(x,r(\omega),s(\omega))d\omega$$ for every $\mbE\tau(x,y,z)\in\Delta^R$, $r\in (M^\Omega)^{|y|}$ and $s\in\mbA^{|z|}$.
\end{lem}
\begin{proof}
Let us denote $K=S_\Delta(M)$. For each measure $\lambda\in\mcS(\Omega,K)$ we define a type $p_\lambda\in S_{\Delta^R}(M^R)$ by setting the value of $p_\lambda$ on a formula $\mbE\tau(x,r,s)$ (as in the statement) to be $$\mbE\tau(x,r,s)^{p_\lambda}\coloneqq\int\tau(x,r(\omega),s(\omega))^qd\lambda(\omega,q)= \int\mbE^{\lambda_\omega}(\tau(x,r(\omega),s(\omega)))d\omega.$$
To see that this defines a type, we may write $\lambda$ as the class of a pair $(t,k)$ in the quotient $(\End(\Omega)\times K^\Omega)\sslash\Aut(\Omega)\simeq\mcS(\Omega,K)$. Suppose first that $t$ is actually in $\Aut(\Omega)$ and that $k$ takes values in a finite set of realized types of $S_\Delta(M)$, so that we may write $k(\omega)=\tp_\Delta(k'(\omega))$ for some $k'\in (M^\Omega)^{|x|}$. Then $p_\lambda$ is a realized type, namely $p_\lambda=\tp_{\Delta^R}(k't^{-1})$. In the general case, $(t,k)$ is a limit of pairs of the previous form, which readily implies that $p_\lambda$ is approximately finitely satisfiable, i.e., a type.

The map $f\colon\omega\mapsto\tau(x,r(\omega),s(\omega))$ is in $L^2(\Omega,C(K))$, and we have $\mbE\tau(x,r,s)^{p_\lambda}=\langle f,\lambda\rangle$. Thus, the map $\theta\colon\lambda\mapsto p_\lambda$ is clearly continuous. By representing measures as we did in the previous paragraph, it is also clear that every realized type in $S_{\Delta^R}(M^R)$ is in the image of $\theta$, hence that $\theta$ is surjective.

Checking that $\theta$ is $G\wr\Omega$-equivariant is a straightforward verification. Finally, if $\lambda\neq\nu$, then there are a set $A\subset\Omega$ and a function $h\in C(K)$ such that $\mbE^\lambda(\chi_Ah)\neq\mbE^\nu(\chi_Ah)$. Now, $h$ is a continuous combination of functions induced by formulas $\varphi(x,b)$ for $\varphi(x,y)\in\Delta$ and $b\in M^{|y|}$. It follows that $\chi_A(\omega)h(q)=\tau(x,r(\omega),s(\omega))^q$ for some appropriate term $\tau$, where we can choose $s=\chi_A$ and $r$ to be a tuple of constant random variables. Hence, $p_\lambda$ and $p_\nu$ differ in the formula $\mbE\tau(x,r,s)$, so we conclude that $\theta$ is injective.
\end{proof}

Next we recall some notions from the theory of Banach representations of dynamical systems as developed by Glasner and Megrelishvili; see, for instance, the survey paper \cite{glamegSurvey}. A \emph{representation} of a (compact, Hausdorff) flow $G\actson X$ on a Banach space $V$ is given by an isometric continuous action $G\actson V$ together with a weak$^*$-continuous map $$\alpha:X\to V^*$$ that is $G$-equivariant with respect to the dual action $G\actson V^*$. The representation is \emph{faithful} if $\alpha$ is injective. If the representation is faithful and $\mcK$ is any class of Banach space containing $V$, then the flow is said \emph{$\mcK$-representable}.

We introduce in addition the following definitions.

\begin{defin} Let $X$ be a $G$-flow and $\mcK$ a class of Banach spaces. We say that $f\in C(X)$ is \emph{$\mcK$-vector-representable} if there are a representation $\alpha$ of $G\actson X$ on a Banach space $V\in\mcK$ and a vector $v\in V$ such that, for all $x\in X$, $$f(x)=\langle v,\alpha(x)\rangle.$$ We denote the family of all $\mcK$-vector-representable continuous functions on $X$ by $B_\mcK(X)$.\end{defin}

We remark that if the class $\mcK$ is closed under Banach subspaces and $G$ is separable, then in the previous definition we can always assume that $V$ is separable. Indeed, it suffices to replace $V$ by the closed subspace generated by $Gv$.

\begin{defin}\label{def:R-closed} Let $\mcK$ be a class of Banach spaces closed under isomorphisms and subspaces. We say that $\mcK$ is \emph{$R$-closed} if, in addition, the following conditions hold.
\begin{enumerate}
\item If $V\in\mcK$, then $L^2(\Omega,V)\in\mcK$.
\item If $X$ is a $\mcK$-representable $G$-flow, then $B_\mcK(X)$ is dense in $C(X)$.
\end{enumerate}
\end{defin}

The main classes of Banach spaces considered in \cite{glamegSurvey}, and in the related works of the same authors, are $R$-closed. Following \cite{glamegSurvey}, we say that a Banach space is \emph{Rosenthal} if it does not contain an isomorphic copy of $\ell^1$.

\begin{lem}
The classes of Hilbert, reflexive, Asplund and Rosenthal Banach spaces are $R$-closed.
\end{lem}
\begin{proof}
We comment on the two conditions of the definition separately.

(1). This is obvious for Hilbert spaces. Asplund spaces can be characterized by the property that $L^2(\Omega,V)^*$ is naturally identified with $L^2(\Omega,V^*)$ (see, for instance, \cite[IV., \textsection 1]{dieuhlBook}), and from this fact the claim follows easily for reflexive and Asplund spaces. For Rosenthal spaces this was proved by Pisier in \cite{pis78}; see also \cite[\textsection 2.2]{cemmenBook}.

(2). For the classes of reflexive, Asplund and Rosenthal spaces, it follows from the works of Glasner and Megrelishvili that every $\mcK$-representable $G$-flow $X$ satisfies $B_\mcK(X)=C(X)$. Particularly, for Rosenthal spaces, this is a consequence of \cite{glamegNotl1}, Theorem~6.7. For the class $\mcK$ of Hilbert spaces, by considering sums and tensor products we see that $B_\mcK(X)$ forms a unital subalgebra of $C(X)$; if $X$ is $\mcK$-representable, then $B_\mcK(X)$ separates points of $X$, hence the Stone--Weierstrass theorem implies that $B_\mcK(X)$ is dense in $C(X)$.
\end{proof}

\begin{rem}\label{rem separating points implies rep} Let $\mcK$ be an $R$-closed class of Banach spaces, and let $X$ be a metrizable $G$-flow. Suppose that the representations of $X$ on Banach spaces in the class $\mcK$ separate points of $X$. Then, $X$ is actually $\mcK$-representable. Indeed, since $\mcK$ is closed under forming $L^2$-spaces and subspaces, it follows that $\mcK$ is closed under $\ell^2$-sums; then, using that $X$ is second countable, we can choose countably many representations separating points and use them to construct a faithful representation on the $\ell^2$-sum of the corresponding spaces, as is done in \cite{megHilb}, Lemma~3.3.
\end{rem}

Let $X$ be a metrizable flow of a Polish group $G$. We construct, for every representation $\alpha\colon X\to V^*$ of $G\actson X$ on a separable Banach space $V$, an induced representation $$\alpha^R\colon\mcS(\Omega,X)\to L^2(\Omega,V)^*$$ of the action $G\wr\Omega\actson\mcS(\Omega,X)$ on the Bochner space $L^2(\Omega,V)$.

The induced action of $G\wr\Omega$ on $L^2(\Omega,V)$ was described in \textsection\ref{ss:Bochner}. As for $\alpha^R$, we may define it by the relation $$\langle f,\alpha^R(\lambda)\rangle=\int\langle f(\omega),\alpha(x)\rangle d\lambda(\omega,x)$$ for $f\in L^2(\Omega,V)$ and $\lambda\in\mcS(\Omega,X)$. By approximating $f$ by continuous functions, it is clear that $\alpha^R$ is continuous, since $\alpha$ is. In addition, it is convenient to define $\alpha^R(\lambda)$ as an element of $L^2_{w^*}(\Omega,V^*)$. Given a measure $m\in\Reg(X)$, the weak$^*$ expectation of $\alpha$ with respect to $m$ is the functional $\mbE^m(\alpha)\in V^*$ defined by $$\langle v,\mbE^m(\alpha)\rangle=\mbE^m(\langle v,\alpha\rangle)=\int\langle v,\alpha(x)\rangle dm(x)$$ for every $v\in V$. Then, given $p\in\mcS(\Omega,X)$, we define $\alpha^R(p)\colon\Omega\to V^*$ by $$\alpha^R(p)(\omega)=\mbE^{p(\omega)}(\alpha).$$ Since $\alpha$ is continuous, its image in $V^*$ is bounded, and this ensures that $\alpha^R(p)\in L^2_{w^*}(\Omega,V^*)$. Clearly, the two definitions of $\alpha^R$ coincide. Also, it is straightforward to check that $\alpha^R$ is $G\wr\Omega$-equivariant.

\begin{theorem}\label{thm general preservation}
Let $G\actson X$ be a continuous action of a Polish group on a compact metrizable space. Let $\mcK$ be an $R$-closed class of Banach spaces. If $G\actson X$ is $\mcK$-representable, then so is $G\wr\Omega\actson\mcS(\Omega,X)$.
\end{theorem}
\begin{proof}
By Remark~\ref{rem separating points implies rep}, it suffices to show that the representations on Banach spaces in the class $\mcK$ separate points of $\mcS(\Omega,X)$. Suppose that $p,q\in\mcS(\Omega,X)$ cannot be separated in this way. In particular, since $\mcK$ is $R$-closed, $\alpha^R(p)=\alpha^R(q)$ for every representation $\alpha\colon X\to V^*$ on a separable Banach space $V\in\mcK$. For such $\alpha$, if $v\in V$, we have then $$\langle v,\mbE^{p(\omega)}(\alpha)\rangle=\langle v,\mbE^{q(\omega)}(\alpha)\rangle$$ for almost every $\omega$.

Since $X$ is $\mcK$-representable, our hypothesis on $\mcK$ ensures that we can find a countable dense family $F\subset C(X)$ consisting of $\mcK$-vector-representable functions. For each $f\in F$, let $\alpha_f$ be a representation of $X$ on a separable Banach space $V_f\in\mcK$ with a vector $v_f\in V_f$ such that $f(x)=\langle v_f,\alpha_f(x)\rangle$ for every $x\in X$. Since $F$ is countable, we have that $\langle v_f,\mbE^{p(\omega)}(\alpha_f)\rangle=\langle v_f,\mbE^{q(\omega)}(\alpha_f)\rangle$ for all $f\in F$ and every $\omega$ in a common full-measure set. That is, $\mbE^{p(\omega)}(f)=\mbE^{q(\omega)}(f)$ for every $f\in F$ and almost every $\omega$. Since $F$ is dense in $C(X)$, it follows that $p(\omega)=q(\omega)$ almost everywhere. That is, $p=q$, and the theorem follows.
\end{proof}

By thinking of $X$ as a type space, the previous result can be thought of as a Banach-theoretic counterpart to the preservation results of model-theoretic properties by randomizations, studied within \cite{benkei,benOntheories,benVapChe}. In the case of $\aleph_0$-categorical theories, by the translation discussed in \cite{iba14}, this correspondence is exact. Indeed, suppose $T$ is separably categorical, and let $\varphi(x,y)$ be any formula. We obtain a new proof of the following, which was shown in \cite{benVapChe}; see Lemma~5.4 therein for the usual definition of (the negation of) NIP in the metric setting.

\begin{cor}
If $\varphi(x,y)$ is NIP for $T$, then $\mbE\llbracket\varphi(x,y)\rrbracket$ is NIP for $T^R$.
\end{cor}
\begin{proof}
Let us fix a model $M$ of $T$ and $G=\Aut(M)$. It follows from \cite[\textsection 3]{iba14} that a formula $\varphi(x,y)$ is NIP if and only if, for some set of formulas $\Delta$ containing $\varphi$, the action $G\actson S_\Delta(M)$ is Rosenthal-representable. In that case, by our previous results, the action $\Aut(M^R)\actson S_{\Delta^R}(M^R)$, which is the same as the action $G\wr\Omega\actson\mcS(\Omega,S_\Delta(M))$, is Rosenthal-representable too. Since $\Delta^R$ contains the formula $\mbE\llbracket\varphi(x,y)\rrbracket$, we deduce that the latter is NIP.
\end{proof}

If instead of considering Rosenthal spaces we consider reflexive spaces, then the same argument yields yet another proof of Corollary~\ref{cor pres stability}.

\noindent\hrulefill

\section{Beautiful pairs of randomizations}\label{s:beautiful}

In this section we study the theory of beautiful pairs of models of a randomized $\aleph_0$-categorical theory. Let us first explain our motivation to do so.

For \emph{classical}, stable, $\aleph_0$-categorical structures, a number of very dissimilar properties turn out be equivalent. For instance, if $M$ is one such structure, $T=\text{Th}(M)$, $G=\Aut(M)$, then each of the following conditions implies the other:
\begin{enumerate}
\item $M$ is $\aleph_0$-stable.
\item $M$ is one-based.
\item The theory $T_P$ of beautiful pairs of models of $T$ is $\aleph_0$-categorical.
\item The Roelcke compactification $R(G)$ is Hilbert-representable.
\end{enumerate}

Within these, the central notion is (2), which has direct, strong model-theoretic consequences for $M$. That (1) is equivalent to (2) is a classical, intricate theorem (see, for instance, Theorem~5.12 in \cite[Ch.~2]{pil96}). The equivalence of (2) and (3) was proved in \cite{bbhSFB14}, and the equivalence of (2) and (4) was shown in \cite{bit15}.

When we pass from classical to continuous logic, the most basic and well-behaved new structures we get fail to be one-based. Thus the question arises of whether there exists an appropriate generalization of this notion to the metric setting. In \cite{bbhSFB14}, the authors propose a generalization (for metric, stable theories) that does hold in some important examples, which they call \emph{SFB} (for \emph{strongly finitely based}). They focus on (metric, stable) $\aleph_0$-categorical theories, and there they show the following: $T$ is SFB if and only if the theory $T_P$ of beautiful pairs of models of $T$ is $\aleph_0$-categorical. We may take this as a definition.

We will point out here a weakness of this proposed generalization, by proving a non-preservation result: the property SFB is not preserved by randomizations. In fact, it fails badly for most randomized theories (even though it does hold for the theory of the measure algebra of~$\Omega$).

\subsection{The theory $T_P$} We recall the basic definitions and facts about beautiful pairs of models of a stable metric theory, and refer to \cite[\textsection 4]{benOnuniform} for more details. An \emph{elementary pair} of models of a theory $T$ consists of a model $M\models T$ together with an elementary substructure $N\prec M$. A \emph{beautiful pair} of models of $T$ is an elementary pair $(M,N)$ such that $N$ is approximately $\aleph_0$-saturated (as per \cite{benusv-d-finiteness}, Definition~1.3) and $M$ is approximately $\aleph_0$-saturated over $N$, that is to say, the structure $M$ augmented with constants for the elements of $N$ is approximately $\aleph_0$-saturated. (We follow the definition of \cite{benOnuniform}, which is broader than the one given in \cite{bbhSFB14}, although both induce the same theory~$T_P$.)

Elementary pairs of models of an $L$-theory $T$ are considered in the language $L_P$, which is $L$ expanded with a predicate~$P$ for the distance to the smaller model of the pair. We denote by $T_P$ the common theory of all beautiful pairs of models of $T$ in this language, and we write $(M,N)\models T_P$ to say that $M$ together with the interpretation $P(x)=d(x,N)$ forms a model of $T_P$.

When $T$ is $\aleph_0$-categorical, it can be shown that any saturated model of $T_P$ is again a beautiful pair; this fact is expressed by saying that the class of beautiful pairs of models of $T$ is \emph{almost elementary}. However, we will work with separable models of $T_P$, which need not be beautiful pairs. To this end, we will use the following general characterization of the models of $T_P$, which follows from the proof of Theorem~4.4 of \cite{benOnuniform}.

\begin{theorem}\label{axioms for TP} Suppose $T$ is a stable $L$-theory whose class of beautiful pairs of models is almost elementary. Let $N\prec M$ be models of $T$. Then, $(M,N)\models T_P$ if and only if the following holds: for every $\epsilon>0$, every finite $z$-tuple $c$ from $M$, every 1-type $p\in S_x(Nc)$ and every finite set of $L$-formulas $\varphi_i(x,yz)$, $i<n$, there is $a\in M$ such that $$|\varphi_i(a,bc)-\varphi_i(x,bc)^p|<\epsilon$$ for every $y$-tuple $b$ in $N$ and each $i<n$. In other words, types over finite expansions of $N$ are approximately finitely realized in $M$ uniformly on the parameters.\end{theorem}

\begin{rem}\label{rem smaller small model} If $(M,N)\models T_P$ and $\tilde{N}\prec N$, then also $(M,\tilde{N})\models T_P$.\end{rem}

\subsection{Separable models of $(T^R)_P$} We will consider two copies of the unit interval $\Omega$, say $\Omega_0$ and $\Omega_1$. Then, $\Omega^2$ will denote the product space $\Omega_0\times\Omega_1$, and $\Omega$ will stand for its factor induced by $\Omega_0$; that is, $\Omega$ will denote the measure space $\Omega_0\times\Omega_1$ restricted to the sub-$\sigma$-algebra generated by the projection $\Omega_0\times\Omega_1\to\Omega_0$. In this way, if $X$ is a subset of a Polish space $Y$, then $X^\Omega$ becomes a subset of $Y^{\Omega^2}$. The measure on each of the spaces $\Omega_0$, $\Omega_1$, $\Omega$ or $\Omega^2$ will still be denoted by $\mu$.

Let us denote $\mbA_\Omega\coloneqq [0,1]^\Omega$ and $\mbA_{\Omega^2}\coloneqq [0,1]^{\Omega^2}$. Hence, $\mbA_\Omega$ is a substructure of $\mbA_{\Omega^2}$. Let $ARV$ denote $\text{Th}(\mbA)$, that is, the theory of $[0,1]$-valued random variables over atomless probability spaces. Finally, we denote by $\mbA_P$ the pair $(\mbA_{\Omega^2},\mbA_\Omega)$, which is a structure in the language of pairs of models of $ARV$.

\begin{prop} The theory $ARV_P$ of beautiful pairs of models of $ARV$ is $\aleph_0$-categorical, and we have $\mbA_P\models ARV_P$.\end{prop}
\begin{proof}
See \cite{bbhSFB14}, Corollary~3.15.
\end{proof}

From now on, we fix an $\aleph_0$-categorical, stable theory~$T$ and a separable model $M\models T$. As before, the theory of beautiful pairs of models of $T$ is $T_P$ and the randomization of $T$ is $T^R$. The theory of beautiful pairs of models of $T^R$ is $(T^R)_P$. Rather than describing the beautiful pairs of models of $T^R$, we are interested in the separable models of $(T^R)_P$.

Given any elementary pair $\mbP$ of models of $T^R$, we can consider the reduct formed by the pair of their auxiliary sorts. This is an elementary pair of models of $ARV$, which we denote by $\mbA_\mbP$, and which we may call the \emph{auxiliary sort of $\mbP$}.

\begin{rem} If $\mbP\models (T^R)_P$, then, clearly, $\mbA_\mbP\models ARV_P$.\end{rem}

Hence, if we have a separable model $\mbP\models (T^R)_P$, then $\mbA_\mbP\simeq\mbA_P$. It follows that we can identify the large model of the pair $\mbP$ with the Borel randomization $M^R$ based on $\Omega^2$ (that is, with main sort $M^{\Omega^2}$ and auxiliary sort $\mbA_{\Omega^2}$) and the small model of the pair $\mbP$ with some substructure $S\prec M^R$ whose auxiliary sort is $\mbA_\Omega$. Thus, in order to classify the separable models $\mbP\models (T^R)_P$ up to isomorphism, we are left to understand the different possibilities for the main sort of $S$.

\begin{notation} From now on, unless otherwise stated, $M^R$ will denote the randomization of $M$ based on $\Omega^2$, as above. Given a submodel $S\prec M^R$ with main sort $S_0$ and auxiliary sort $\mbA_\Omega$, we will denote by $(M^{\Omega^2},S_0)_{\mbA_P}$ the elementary pair $(M^R,S)$ of models of $T^R$.\end{notation}

It is natural to expect that, if $(M,N)$ is a model of $T_P$, then $(M^{\Omega^2},N^\Omega)_{\mbA_P}$ should be a model of $(T^R)_P$. This is correct, but we will prove that this does in no way exhaust the models of $(T^R)_P$ (except in trivial cases). Given $h\in\End(M)^{\Omega^2}$, let $$\mbP_h\coloneqq (M^{\Omega^2},S_h)_{\mbA_P},$$ where $S_h\coloneqq\{hs:s\in M^\Omega\}$. The following is the main result of this section.

\begin{theorem}\label{thm models of TRP}
Let $\mbP$ be a separable elementary pair of models of $T^R$. Then, $\mbP\models (T^R)_P$ if and only if $\mbA_\mbP\simeq\mbA_P$. Moreover, in that case, $\mbP\simeq\mbP_h$ for some $h\in\End(M)^{\Omega^2}$.
\end{theorem}

In particular, for any $N\prec M$, the pair $(M^{\Omega^2},N^\Omega)_{\mbA_P}$ is a model of $(T^R)_P$. Clearly, this leads to non-isomorphic models of $(T^R)_P$, as long as there exists a pair $N\prec M$ with $N\subsetneq M$. Indeed, for such $N$, the pairs $(M^{\Omega^2},N^\Omega)_{\mbA_P}$ and $(M^{\Omega^2},M^\Omega)_{\mbA_P}$ are distinct models of $(T^R)_P$. Conversely, if $M$ does not have proper elementary substructures (which happens if and only if $M$ is compact), then $\End(M)=\Aut(M)$, and in that case there is only one model of $(T^R)_P$ up to isomorphism.

\begin{cor}\label{cor SFB not preserved} The theory $(T^R)_P$ is not $\aleph_0$-categorical, unless $T$ is the theory of a compact structure.\end{cor}

As said before, this shows that the property SFB defined in \cite{bbhSFB14} is not preserved by randomizations.

\begin{example} The theory of the randomization of a countable set (with no further structure) is not SFB. On the other hand, it is $\aleph_0$-stable, and the Roelcke compactification of $S_\infty\wr\Omega$ is Hilbert-representable.\end{example}

We turn to the preliminaries for the proof of Theorem~\ref{thm models of TRP}. Let $S\prec M^R$ be a submodel with auxiliary sort equal to $\mbA_\Omega$. We have discussed the elementary substructures of $M^R$ in Remark~\ref{rem submodels of M^R}. We know that $S$ is the image of an endomorphism $ht^*$ of $M^R$, where $h\in\End(M)^{\Omega^2}$ and $t\in\End(\Omega^2)$.

Then the image of $t^*\in\End(\mbA_{\Omega^2})$ must be $\mbA_\Omega$, and hence the main sort of $S$ is $$\{ht^*r:r\in M^{\Omega^2}\}=\{hs:s\in M^\Omega\}=S_h.$$ It follows that every model of $(T^R)_P$ is isomorphic to $\mbP_h$ for some $h\in\End(M)^{\Omega^2}$. We want to see that all these are indeed models of $(T^R)_P$.

Since $S_h$ is an elementary substructure of $S_1=M^\Omega$ (when endowed with the auxiliary sort~$\mbA_\Omega$), by Remark~\ref{rem smaller small model} it is enough to prove that $\mbP_1\models (T^R)_P$. In fact, we will give a proof of the following (which, incidentally, does not use the $\aleph_0$-categoricity of $M$).

\begin{prop}\label{prop models of TRP} Let $N\prec M$ be any elementary substructure. Then, $(M^{\Omega^2},N^\Omega)_{\mbA_P}\models (T^R)_P$.\end{prop}

We need some preliminary lemmas.

\begin{lem}\label{lem family of finite measures} Let $A\subset\Omega^2$ be a Borel subset. Suppose we are given a compact metrizable space $X$ together with a weak$^*$-measurable family $(p_{\omega_0})_{\omega_0\in\Omega_0}$ of finite Borel measures on $X$, $$\omega_0\in\Omega_0\mapsto p_{\omega_0}\in C(X)^*,$$ with $p_{\omega_0}(X)=\mu(A_{\omega_0})$; here, $A_{\omega_0}\subset\Omega_1$ is the section of $A$ at $\omega_0$. Then, there is a measurable function $h\colon A\to X$ such that
\begin{equation}\label{eq lem fam of fin meas}\tag{*}
\int_X fdp_{\omega_0}=\int_{A_{\omega_0}}f\circ h_{\omega_0}d\omega_1
\end{equation}
for every $f\in C(X)$ and almost every $\omega_0\in\Omega_0$; here, $h_{\omega_0}(\omega_1)=h(\omega_0,\omega_1)$.\end{lem}
\begin{proof}
We first note that for each $\omega_0$ we can find a measurable function $h_{\omega_0}\colon A_{\omega_0}\to X$ satisfying (\ref{eq lem fam of fin meas}). Indeed, the measure algebra of $(X,p_{\omega_0})$ is separable, so it can be embedded in the measure algebra of $(A_{\omega_0},\mu)$, since the latter is atomless. By duality we get a measure preserving map $h_{\omega_0}\colon A_{\omega_0}\to X$, which is what we wanted.

We have to ensure that we can choose $h_{\omega_0}$ in a measurable way. We consider partial measurable functions from $\Omega_1$ to $X$: say $S=(X\cup\{*\})^{\Omega_1}$, and for $h_0\in S$ define $s(h_0)=h_0^{-1}(X)$ to be the support of $h$. Now, it is not difficult to see that the set $$E=\{(\omega_0,h_0)\in\Omega_0\times S:s(h_0)=A_{\omega_0}\text{ and }\mbE^{p_{\omega_0}}(f)=\mbE^\mu(\chi_{A_{\omega_0}}f\circ h_0)\text{ for all }f\in C(X)\}$$ is Borel. By the previous paragraph, the projection of $E$ to $\Omega_0$ is $\Omega_0$, so from the Jankov--von Neumann uniformization theorem (see \cite[\textsection 18.A]{kechrisDST}; note that analytic sets are Lebesgue measurable) we obtain a measurable function $h\colon\Omega_0\to S$ such that $h(\omega_0)\in E_{\omega_0}$ almost surely. By the natural identification $S^{\Omega_0}\simeq (X\cup\{*\})^{\Omega_0\times\Omega_1}$, this induces a measurable function $h\colon A\to X$ satisfying the requirements of the lemma.
\end{proof}

\begin{lem}\label{stability facts} Let $M$ be a separable stable metric structure and $\varphi(x,y)$ be any formula.
\begin{enumerate}
\item The uniform pseudo-distance $$d_\varphi(p,q)=\sup_{b\in M^{|y|}}|\varphi(x,b)^p-\varphi(x,b)^q|$$ on the type space $S_x(M)$ is separable. Moreover, every open set for the distance $d_\varphi$ is Borel measurable for the logic topology of $S_x(M)$.
\item For every $\epsilon>0$ there is a natural number $m\in\mbN$ such that for every $q\in S_x(M)$ there exists a sequence $(a_l)_{l\in\mbN}\subset M^{|x|}$ with the property that, for any tuple $b\in M^{|y|}$, $$|\{l\in\mbN:|\varphi(a_l,b)-\varphi(x,b)^q|\geq\epsilon\}|\leq m.$$
\end{enumerate}
\end{lem}
\begin{proof}
(1). The first assertion follows from the separability of $M$ and the definability of types in stable theories. See for example \cite{ben13}, Corollary 4. The second follows directly from the separability of $M$.

(2). By stability, there is a natural number $m$ such that, whenever $a_l$ is an indiscernible sequence with limit type $q$ and $b$ is any tuple, the counting inequality displayed in the statement is satisfied. The following argument by compactness shows that there are a finite set of formulas $\Delta$ and a positive number $\delta$ such that the same is true whenever $a_l$ is a \emph{$\Delta$-$\delta$-indiscernible} sequence (in the sense defined in \cite[\textsection 1.4]{iba14}) converging to~$q$. Suppose this does not hold. Let $d_q\varphi$ be the $\varphi$-definition of~$q$. Then, for any finite $\Delta$ and $\delta>0$ we can find a large finite $\Delta$-$\delta$-indiscernible sequence $a_l$ and an element $b$ such that, for odd $l$, $$|\varphi(a_l,b)-d_q\varphi(b)|\geq\epsilon,$$ whereas, for even $l$, $$|\varphi(a_l,b)-d_q\varphi(b)|<\epsilon/2.$$ By compactness, we get an infinite indiscernible sequence $a_l$ and an element $b$ with this property, which yields a contradiction. The bound $m$ can be chosen uniformly in $q$ since types are uniformly definable.

Now, if $q\in S_x(M)$ is any type, we can take any sequence in $M$ converging to $q$ and extract by Ramsey's theorem a $\Delta$-$\delta$-indiscernible subsequence $a_l$. The lemma follows.
\end{proof}

\begin{proof}[Proof of Proposition \ref{prop models of TRP}] We have to check the condition of Theorem \ref{axioms for TP} for the pair $(M^{\Omega^2},N^\Omega)_{\mbA_P}$. For simplicity, we will only check it for basic formulas of the form $\mbE\llbracket \varphi_i(x,yz)\rrbracket$, so in particular the tuples $xyz$ will only contain variables from the \emph{main sort}. It is not difficult to see that this is actually enough. Moreover, it suffices to check the condition for $z$-tuples consisting of \emph{simple} elements of $M^{\Omega^2}$, that is, random variables of finite range.

So we fix formulas $\mbE\llbracket\varphi_i(x,yz)\rrbracket$, $i<n$, we fix a simple $z$-tuple $t$, and a type $p\in S^{T^R}_x(N^\Omega t)$. Let $C\subset M^{|z|}$ be the (finite) range of~$t$. We fix $c\in C$, then set $A_c=t^{-1}(c)\subset\Omega^2$. By taking any extension of $p$ to a type over $M^{\Omega^2}$, we may apply Lemma~\ref{lem types as random variables} and see $p$ as a random variable $p\colon\Omega^2\to\Reg(S^T_x(M))\to\Reg(S^T_x(Nc))$.

We consider the type space $X=S^T_x(Nc)$. For each $\omega_0\in\Omega_0$ and $f\in C(X)$, say with $f$ induced by a formula $\varphi(x,bc)$, we set $$\langle f,p_{\omega_0}\rangle=\int_{(A_c)_{\omega_0}}\mbE^{p(\omega_0,\omega_1)}(\varphi(x,bc))d\omega_1.$$ This defines a finite measure $p_{\omega_0}\in C(X)^*$ with $p_{\omega_0}(X)=\mu((A_c)_{\omega_0})$. Applying Lemma~\ref{lem family of finite measures}, we get a measurable function $h_c\colon A_c\to X$ that satisfies $$\int_{(A_c)_{\omega_0}}\mbE^{p(\omega_0,\omega_1)}(\varphi(x,bc))d\omega_1=\int_{(A_c)_{\omega_0}}\varphi(x,bc)^{h_c(\omega_0,\omega_1)}d\omega_1$$ for almost every $\omega_0$ and every $Nc$-formula $\varphi(x,bc)$.

Let $\epsilon>0$. Using the first item of Lemma \ref{stability facts}, we can find a countable set $J\subset S^T_x(M)$ and measurable functions $j_c\colon A_c\to J$ for each $c\in C$ such that $$|\varphi_i(x,bc)^{h_c(\omega)}-\varphi_i(x,bc)^{j_c(\omega)}|\leq\epsilon$$ for every $b\in N^{|y|}$, $c\in C$, $i<n$ and $\omega\in\Omega^2$. Next, we apply the second item of the lemma to get sequences $(a_q^l)_{l\in\mbN}\subset M$ for each $q\in J$ and a natural number $m\in\mbN$ such that, for each of the formulas $\varphi_i(x,yz)$, $i<n$, and for any $b\in N^{|y|}$, $c\in C$ and $q\in J$, we have $$|\{l\in\mbN:|\varphi_i(a_q^l,bc)-\varphi_i(x,bc)^q|\geq\epsilon\}|\leq m.$$ In any case, $|\varphi_i(a_q^l,bc)-\varphi_i(x,bc)^q|\leq 1$ since we assume formulas are $[0,1]$-valued.

Take $k\in\mbN$ with $m/k<\epsilon$. For each $c\in C$ and $q\in J$ we can choose a Borel partition $\{A_{c,q}^l\}_{l<k}$ of $A_{c,q}\coloneqq (h_c)^{-1}(q)\subset A_c$ such that, for almost every $\omega_0\in\Omega_0$, we have $$\mu((A_{c,q}^l)_{\omega_0})=\frac{1}{k}\mu((A_{c,q})_{\omega_0}).$$ Finally, we define $r:\Omega^2\to M$ by $$r=\sum_{c\in C,q\in J,l<k}a_q^l\chi_{A_{c,q}^l}.$$

In this way, for any $i<n$ and any tuple of random variables $s\in (N^{|y|})^\Omega$ (which depends only on the variable $\omega_0\in\Omega_0$), we get
\begin{align*}
\left|\mbE\llbracket\varphi_i(r,st)\rrbracket-\mbE\llbracket\varphi_i(x,st)\rrbracket^p\right| & \leq \sum_{c\in C}\left|\int_{A_c}\varphi_i(r(\omega_0,\omega_1),s(\omega_0)c)-\mbE^{p(\omega_0,\omega_1)}(\varphi_i(x,s(\omega_0)c))d\mu\right|
\end{align*}
\vspace{-.3cm}
\begin{align*}
\hspace{2.3cm} & \leq \sum_{c\in C}\int_{\Omega_0}\left|\int_{(A_c)_{\omega_0}}\varphi_i(r(\omega_0,\omega_1),s(\omega_0)c)-\mbE^{p(\omega_0,\omega_1)}(\varphi_i(x,s(\omega_0)c))d\omega_1\right| d\omega_0\\
& = \sum_{c\in C}\int_{\Omega_0}\left|\int_{(A_c)_{\omega_0}}\varphi_i(r(\omega_0,\omega_1),s(\omega_0)c)-\varphi_i(x,s(\omega_0)c)^{h_c(\omega_0,\omega_1)}d\omega_1\right| d\omega_0\\
& \leq \epsilon + \sum_{c\in C}\int_{\Omega_0}\int_{(A_c)_{\omega_0}}\left|\varphi_i(r(\omega_0,\omega_1),s(\omega_0)c)-\varphi_i(x,s(\omega_0)c)^{j_c(\omega_0,\omega_1)}\right| d\omega_1 d\omega_0\\
& = \epsilon + \sum_{c\in C,q\in J,l<k}\int_{\Omega_0}\int_{(A_{c,q}^l)_{\omega_0}}\left|\varphi_i(a_q^l,s(\omega_0)c)-\varphi_i(x,s(\omega_0)c)^q\right|d\omega_1 d\omega_0\\
& = \epsilon + \sum_{c\in C,q\in J}\int_{\Omega_0}\frac{1}{k}\mu((A_{c,q})_{\omega_0})\sum_{l<k}\left|\varphi_i(a_q^l,s(\omega_0)c)-\varphi_i(x,s(\omega_0)c)^q\right|d\omega_0\\
& \leq \epsilon + \sum_{c\in C,q\in J}\int_{\Omega_0}\left(\frac{m}{k}+\epsilon\right)\mu((A_{c,q})_{\omega_0})d\omega_0\\
& <3\epsilon.
\end{align*}
We have thus verified the condition of Theorem \ref{axioms for TP} for $(M^{\Omega^2},N^{\Omega})$.
\end{proof}

Together with the discussion preceding Proposition~\ref{prop models of TRP}, this completes the proof of Theorem~\ref{thm models of TRP}.

\begin{rem}\label{rem beautiful pairs rando} Suppose $M$ is a classical structure, and that $(M,N)$ is in fact a beautiful pair of (countable) models of $T$. In this case, the argument of Proposition~\ref{prop models of TRP} becomes much simpler, and yields more. Indeed, resuming the argument after the third paragraph, for every type $q\in S^T_x(Nc)$ we can choose a realization $a_c(q)\in M^{|x|}$. Then we take $r=\sum_{c\in C}(a_c\circ h_c)\chi_{A_c}$, and we see readily that $r$ is a realization of $p$. It follows that $(M^{\Omega^2},N^\Omega)_{\mbA_P}$ is a beautiful pair of models of $T^R$.\end{rem}

By using the same ideas we obtain the following, which is the metric generalization of Theorem~4.1 of \cite{benkei}.

\begin{theorem}
Let $T$ be a metric theory in a countable language. If $T$ is $\aleph_0$-stable, then so is $T^R$.
\end{theorem}
\begin{proof}
Since $T$ is $\aleph_0$-stable, there is an elementary pair $(M,N)$ of separable models of $T$ such that $N$ realizes every type over the empty set (possibly in countably many variables) and $M$ realizes every type over $N$. (In fact, there is even a separable beautiful pair, as can be seen using \cite{benusv-d-finiteness}, Proposition~1.16.) That is, the canonical map $\pi\colon M^{|x|}\to S^T_x(N)$ is surjective. We claim that it admits a Borel selector. By the same argument of Lemma~\ref{Borel selector}, it suffices to check that $\pi(U)$ is Borel for every open set $U\subset M^{|x|}$. We will check it first for the \emph{metric topology} of $S^T_x(M)$. Let $a\in U$ and take $\epsilon>0$ such that $b\in U$ whenever $d(a,b)<\epsilon$. If $d(\pi(a),q)<\epsilon$, then by saturation there are $a',b\in M^{|x|}$ such that $\pi(a')=\pi(a)$, $\pi(b)=q$ and $d(a',b)<\epsilon$. Hence, $\pi(U)$ is open for the metric topology. Now, since $T$ is $\aleph_0$-stable and $N$ is separable, the metric topology of the type space is Polish, and it follows that $\pi(U)$ is an $F_\sigma$ set for the usual compact topology. We deduce that there is a Borel selector for $\pi$, say $a\colon S^T_x(N)\to M$.

Now we may proceed as in Remark~\ref{rem beautiful pairs rando} (ignoring the variable $z$), to show that $M^{\Omega^2}$ realizes every type over $N^\Omega$. Indeed, given $p\in S^{T^R}_x(N^\Omega)$, there is $h\colon\Omega^2\to S^T_x(N)$ such that $\int\mbE^{p(\omega_0,\omega_1)}(\varphi(x,b))d\omega_1=\int\varphi(x,b)^{h(\omega_0,\omega_1)}d\omega_1$ for any $\omega_0,b,\varphi$. If we define $r=a\circ h\in M^{\Omega^2}$, then $r$ is a realization of $p$.

Similarly, $N^\Omega$ realizes every type over the empty set. Then, the pair $(M^{\Omega^2},N^\Omega)_{\mbA_P}$ witnesses that $T^R$ is $\aleph_0$-stable.
\end{proof}

It had already been observed in \cite{bbhSFB14} that an $\aleph_0$-stable, $\aleph_0$-categorical theory need not be SFB. Namely, the theory $AL_pL$ of atomless $L^p$ Banach lattices (for any fixed $p\in [1,\infty)$) is $\aleph_0$-stable and admits only one separable model up to isomorphism, but the corresponding theory $AL_pL_P$ of beautiful pairs admits exactly two non-isomorphic models. To this example we can now add any randomized theory $T^R$ where $T$ is an $\aleph_0$-stable, $\aleph_0$-categorical theory with a non-compact model.

Nevertheless, they point out in \cite{bbhSFB14} that the theory $AL_pL_P$ is $\aleph_0$-categorical \emph{up to arbitrarily small perturbations of the predicate $P$}. For a general theory $T$, this means that for every $\epsilon>0$ and any two separable models $(M,N),(M',N')\models T_P$, there exists an isomorphism $\rho\colon M\to M'$ such that $$|d(x,N)-d(\rho(x),N')|\leq\epsilon$$ for every $x\in M$. If $T_P$ has this property, let us say that $T$ is \emph{approximately SFB}. Then, it was conjectured in \cite{bbhSFB14} that an $\aleph_0$-stable, $\aleph_0$-categorical theory should be approximately SFB. Our new examples give an interesting family to test this conjecture.

\begin{question} Let $T$ be an SFB theory, for instance, a classical $\aleph_0$-stable, $\aleph_0$-categorical theory. Is it true that the randomized theory $T^R$ is approximately SFB?\end{question}

\subsection{Automorphisms of pairs}

Since this work originated in the study of automorphism groups of randomized structures, let us finish with a description of the automorphism group of a pair $(M^{\Omega^2},N^\Omega)_{\mbA_P}$, in the fashion of Theorem~\ref{thm G^R}.

We begin by describing the automorphism group of the auxiliary sort, $H^R_P\coloneqq\Aut(\mbA_P)$. Via the natural isomorphism $[0,1]^{\Omega^2}\simeq ([0,1]^{\Omega_1})^{\Omega_0}$, we see that the structure $\mbA_P$ (a model of $ARV_P$) can be in a sense identified with the randomization $(\mbA_{\Omega_1})^R$ (a model of $ARV^R$): they are bi-interpretable. In particular, they have the same automorphism groups: $$H^R_P=\Aut(\Omega_1)\wr\Omega_0.$$

Now we let $M$ be a separably categorical, stable structure, and we take $G=\Aut(M)$. We fix an elementary substructure $N\prec M$, and we consider the automorphism group of the pair, $$G_P\coloneqq\Aut(M,N),$$ which is the subgroup of all $g\in G$ that preserve the predicate $P(x)=d(x,N)$ (equivalently: that fix $N$ setwise). Similarly, we let $$G^R_P\coloneqq\Aut((M^{\Omega^2},N^\Omega)_{\mbA_P})$$ be the automorphism group of the induced model of $(T^R)_P$, which is the subgroup of $\Aut(M^R)$ (=$G\wr\Omega^2$) fixing $N^\Omega$ and $\mbA_\Omega$ setwise. Furthermore, we consider the subgroup $$G_P^*\coloneqq (G_P)^{\Omega^2}\cap G^R_P.$$

\begin{lem} We have $G_P^*=\{g\in G^{\Omega^2}:g|_{N^\Omega}\in\Aut(N)^\Omega\}$. Moreover, if $g\in G^R_P$ is the identity on the auxiliary sort $\mbA_P$, then $g\in G_P^*$.\end{lem}
\begin{proof} Let $N^R$ denote the smaller model of the pair $(M^{\Omega^2},N^\Omega)_{\mbA_P}$, that is, $N^\Omega$ together with the auxiliary sort $\mbA_\Omega$. If $g\in G^R_P$, then the restriction of $g$ to the $N^R$ is an automorphism of $N^R$, that is, $g|_{N^R}\in\Aut(N)\wr\Omega$. If moreover $g$ is the identity on $\mbA_P$ (which is the case if $g\in G_P^*$), then $g\in G^{\Omega^2}$ and $g|_{N^R}$ is the identity on $\mbA_\Omega$, so $g|_{N^R}\in\Aut(N)^\Omega$. 

Conversely, if we have $g\in G^{\Omega^2}$ and $g|_{N^\Omega}\in\Aut(N)^\Omega$, then $g$ fixes $N^\Omega$ setwise (and $\mbA_\Omega$ pointwise), so $g\in G^R_P$. Also, $g\in (G_P)^{\Omega^2}$. Indeed, let $b$ be an element in $N$, which we may see as a constant element of $M^{\Omega^2}$. Since $g\in G^R_P$, we have $d(gb,N^\Omega)=d(b,N^\Omega)=0$, hence $g(\omega)(b)\in N$ almost surely. By the separability of $N$, this is true for every $b\in N$ and every $\omega$ in a common full measure set. Similarly for $g^{-1}$. Thus, $g(\omega)\in G_P$ almost surely.\end{proof}

\begin{cor}\label{G^RP} $G^R_P\simeq G_P^*\rtimes (\Aut(\Omega_1)\wr\Omega_0)$ as topological groups.\end{cor}
\begin{proof}
As in the proof of Theorem~\ref{thm G^R}, the moreover part of the previous lemma shows that $G_P^*$ is the normal complement of $H_{RP}$, which is what we want.
\end{proof}

The previous description can be used to show that the action $G^P_R\actson (M^{\Omega^2},N^\Omega)_{\mbA_P}$ is not approximately oligomorphic if $M$ is not compact (adapting the method of proof of Proposition~\ref{wreath on rando is app olig}, but with the opposite conclusion; for this, Lemma~\ref{G times AutOmega on XOmega} is useful). This gives an alternative proof of Corollary~\ref{cor SFB not preserved}, modulo showing that $(M^{\Omega^2},N^\Omega)_{\mbA_P}\models (T^R)_P$ for some $N\prec M$ (which is easier by assuming $(M,N)\models T_P$). This was actually our original proof of Corollary~\ref{cor SFB not preserved}.

\noindent\hrulefill

\bibliographystyle{amsalpha}
\bibliography{biblio}

\end{document}